\documentclass[english, 12pt,leqno]{amsart}

\usepackage{a4,amsfonts,amsmath,amssymb,enumitem,color,url}

\newtheorem{thm}{Theorem}[section]
\newtheorem{cor}[thm]{Corollary}
\newtheorem{lem}[thm]{Lemma}
\newtheorem{prop}[thm]{Proposition}

\theoremstyle{definition}
\newtheorem{defn}[thm]{Definition}
\newtheorem{exa}[thm]{Example}
\newtheorem{rem}[thm]{Remark}

\numberwithin{equation}{section}

\newcommand{\N}{\mathbf N}
\newcommand{\Z}{\mathbf Z}

\newcommand{\R}{\mathbf R}
\newcommand{\C}{\mathbf C}
\newcommand{\B}{\mathcal B}
\newcommand{\Hi}{\mathcal H}
\newcommand{\Li}{\mathcal L}
\newcommand{\Ki}{\mathcal K}
\newcommand{\Ui}{\mathcal U}
\newcommand{\Vi}{\mathcal V}
\newcommand{\sca}{\hskip-.05cm\mid\hskip-.05cm}

\title
[Spectral multiplicity functions of adjacency operators of graphs]
{Spectral multiplicity functions
\\
of adjacency operators of graphs
\\
and cospectral infinite graphs}



\thanks{The author acknowledges support of the Swiss NSF grant 200020-20040}

\date{23 February 2024}

\author{Pierre de la Harpe}

\address{Pierre de la Harpe, Section de math\'ematiques,
Universit\'e de Gen\`eve, 
\newline
Uni Dufour,
24 rue du G\'en\'eral Dufour, Case postale 64, 1211 Gen\`eve 4, Suisse.}
\email{Pierre.delaHarpe@unige.ch}

\begin{document}

\begin{abstract}
The adjacency operator of a graph has a spectrum
and a class of scalar-valued spectral measures
which have been systematically analyzed;
it also has a spectral multiplicity function
which has been less studied.
The first purpose of this article is to review
some examples of infinite graphs
for which the spectral multiplicity function
of the adjacency operator has been determined.
The second purpose of this article is to show explicit examples
of infinite connected graphs
which are cospectral,
i.e., which have unitarily equivalent adjacency operators,
and explicit examples of infinite connected graphs
which are uniquely determined by their spectrum.
\end{abstract}

\maketitle

\section{Introduction}
\label{SectionIntro}

Let $G$ be a \textbf{graph} with vertex set $V$ and edge set $E$.
Here graphs are without loops and multiple edges (except in Section~\ref{SectionSurGrNP}),
and $E$ is a set of unordered pairs of vertices.
The degree $\deg (u)$ of a vertex $u \in V$ is the number of edges incident to $u$.
We assume below that $V$ is non-empty, countable (infinite or finite),
and that $G$ is of bounded degree, i.e., that $\max_{u \in V} \deg (u) < \infty$.
Let $\ell^2(V)$ denote the complex Hilbert space of functions $\xi \, \colon V \to \C$
such that $\sum_{u \in V} \vert \xi(u) \vert^2 < \infty$.
It has a canonical orthonormal basis $\left( \delta_u \right)_{u \in V}$;
the value of $\delta_u \in \ell^2(V)$ is~$1$ at~$u$ and~$0$ at other vertices.
The \textbf{adjacency operator} of $G$ is
the bounded self-adjoint linear operator
$A_G$ on $\ell^2(V)$ defined by
$$
(A_G \xi)(u) = \sum_{v \in V, \hskip.1cm \{u, v\} \in E} \xi(v)
\hskip.5cm \text{for all} \hskip.2cm
\xi \in \ell^2(V) 
\hskip.2cm \text{and} \hskip.2cm
u \in V .
$$
Adjacency operators appear in the theory of both finite graphs and infinite graphs.
From the vast literature, we quote
\cite{CoSi--57}, \cite{CvDS--80}, \cite{Bigg--93}, \cite{GoRo--01}, \cite{BrHa--12}
for finite graphs, and
\cite{Kest--59},
\cite{Moha--82},
\cite{GoMo--88}, \cite{MoWo--89},
\cite{BaGr--00},
\cite{DuGr--20}, \cite{GrNP--22}
for infinite graphs.
\par

As for any self-adjoint operator, the Hahn--Hellinger Multiplicity Theorem
implies that $A_G$ is characterized up to unitary equivalence
by three invariants (see Section~\ref{SectionHH}):
\begin{enumerate}
\item[---]
the \textbf{spectrum} $\Sigma (A_G)$, also called the spectrum of $G$,
which is a nonempty compact subset of $\R$,
\item[---]
a \textbf{scalar-valued spectral measure} $\mu_G$
which is a finite Borel measure on $\Sigma (A_G)$,
well-defined up to equivalence,
sometimes viewed as a measure on $\R$ with closed support $\Sigma (A_G)$,
\item[---]
the \textbf{spectral multiplicity function} $\mathfrak m_G$,
which is a measurable function from $\Sigma (A_G)$ to $\{1, 2, \hdots, \infty\}$,
well defined up to equality $\mu_G$-almost everywhere.
\end{enumerate}
We define the \textbf{marked spectrum} of $G$
to be the triple $(\Sigma(A_G), [\mu_G], \mathfrak m_G)$,
where $[\mu_G]$ denotes the class of a spectral valued measure $\mu_G$.
Two graphs are \textbf{cospectral} if they have the same marked spectrum.
A graph $G$ is \textbf{determined by its marked spectrum}
if any graph of bounded degree with the same marked spectrum
is isomorphic to $G$.
\par

Let $G = (V, E)$ be a graph which is finite,
or more generally a graph such that $\ell^2(V)$
has an orthonormal basis of eigenvectors of $A_G$,
for example the Cayley graph of a lamplighter group
as in \cite{GrZu--01} and \cite{BaWo--05}.
The scalar-valued spectral measures of $A_G$ are precisely
the measures which charge every eigenvalue of $A_G$,
so that the meaningful part of the marked spectrum of $G$
reduces to the pair $(\Sigma(A_G), \mathfrak m_G)$.
For such a graph, the spectral multiplicity functions
can be defined as in finite graph theory:
$\mathfrak m_G(x) = \dim \ker(x {\rm Id} - A_G)$ for all $x \in \Sigma (A_G)$.
For more general graphs, see Definition~\ref{DefMultF}.
\par

For finite graphs, spectra and multiplicities of eigenvalues
have been studied intensively.
For infinite graphs,
spectra of adjacency operators have attracted a lot of attention,
but in contrast spectral measures a bit less,
and spectral multiplicity functions even less
(even if there are precise computations of multiplicities
for some classes of graphs,
for example for sparse trees \cite{Breu--07}).
\par

The first purpose of this article is to review
a small number of examples of infinite connected graphs $G = (V, E)$
for which the spectral multiplicity function
of~$A_G$ has been determined.
All this is well-known to experts,
but we did not find good references in the literature.
In Section~\ref{SectionHH},
we review various kinds of multiplication operators,
the Hahn--Hellinger Multiplicity Theorem,
and the definition of the spectral multiplicity function
for a bounded self-adjoint operator.
In Propositions~\ref{infray}, \ref{infline}, and \ref{lattices}, we show:

\begin{prop}
The adjacency operator of the infinite ray $R$ has
spectrum $\mathopen[ -2, 2 \mathclose]$,
scalar-valued spectral measure equivalent to Lebesgue measure,
and uniform multiplicity one.
\par
The adjacency operator of the infinite line $L$
has spectrum $\mathopen[ -2, 2 \mathclose]$,
scalar-valued spectral measure equivalent to Lebesgue measure,
and uniform multiplicity two.
\par
For $d \ge 2$, the adjacency operator of the lattice $L_d$ 
has spectrum $\mathopen[ -2d, 2d \mathclose]$,
scalar-valued spectral measure equivalent to Lebesgue measure,
and infinite uniform multiplicity.
\end{prop}

Section~\ref{SectionSSRTree}
is a study of spherically symmetric rooted trees.
For the particular case of regular trees,
we need in Section~\ref{SectionRegTrees}
to recall results on operators defined by infinite Jacobi matrices.
In Propositions~\ref{propRootedTInfmult} and~\ref{propTInfmult},
we show:

\begin{prop}
For $d \ge 2$, the adjacency operator of the
infinite regular rooted tree $T_d^{{\rm root}}$ of branching degree $d$
has spectrum $\mathopen[ -2 \sqrt d, 2 \sqrt 2 \mathclose]$,
scalar-valued spectral measure equivalent to Lebesgue measure,
and infinite uniform multiplicity.
\par
For $d \ge 3$, the adjacency operator of the
regular rooted tree $T_d$ of degree $d$
has spectrum $\mathopen[ -2 \sqrt{d-1}, 2 \sqrt{d-1} \mathclose]$,
scalar-valued spectral measure equivalent to Lebesgue measure,
and infinite uniform multiplicity.
\end{prop}

Examples of cospectral finite graphs date back to the
very first papers in spectral graph theory.
They include a pair of graphs with $5$ vertices,
a pair of connected graphs with $6$ vertices,
a pair of trees with $8$ vertices (already in \cite{CoSi--57}),
and pairs of regular connected graphs with $10$ vertices;
for these and much more, see \cite{BrHa--12} and \cite{HaSp--04}.
It is striking that examples of cospectral pairs
appear that early in spectral graph theory.
In contrast, the study of
the spectrum of the Laplacian of geometric objects
like bounded open domains in Euclidean spaces
goes back to \cite{Weyl--11},
and the question of existence of cospectral plane domains
(rather called isospectral plane domains)
was open for a long time,
indeed from before \cite{Kac--66},
until the discovery of explicit examples
of cospectral plane domains \cite{GoWW--92}.
\par

The second purpose of this article is to show explicit examples
of cospectral infinite connected graphs.
To our knowledge, such examples do not appear explicitly in the literature.
As an immediate consequence of the two previous propositions,
we have Corollaries~\ref{ExampleCospectrauxRT} and~\ref{ExampleCospectrauxT}:
 
\begin{cor}
For any integer $d \ge 2$, 
the graphs $L_d$, $T_{d^2}^{{\rm root}}$ and $T_{d^2+1}$ are cospectral.
\end{cor}

Note that $L_d$ and $T_{d^2+1}$ are Cayley graphs.
Further examples of multiplets of cospectral spherically symmetric rooted trees
are shown in Example~\ref{ExSuitedperio}.
The final Section~\ref{SectionSurGrNP}
is a very short account
of an uncountable family of cospectral Schreier graphs, from \cite{GrNP--22}. 
\par

Our third purpose is to show examples of graphs determined by their spectra.
There are well-known finite graphs determined by their spectra:
finite paths, cycles, complete graphs $K_n$, complete bipartite graphs $K_{n,n}$, 
triangular graphs $T(n)$ with $n \ne 8$; to cite but a few.
For some experts
``it seems more likely that almost all graphs are determined by their spectrum,
than that almost all graphs are not'';
see \cite[Chapter~14, and in particular Section~14.4]{BrHa--12}.
Some finite graphs are determined by their spectra \emph{among connected graphs},
but not among all finite graphs; this is the case for
finite graphs $G$ with $\Vert A_G \Vert \le 2$ \cite{CvGu--75}.
For infinite graphs, we have
Propositions~\ref{propnormAis2} and~\ref{propRischarac}:

\begin{prop}
\label{PropIntroDetermined}
The infinite graph $R$ is determined by its marked spectrum.
\par

Each of the following three graphs is determined by its marked spectrum
among connected graphs of bounded degree:
$R$, the graph $D_\infty$ of Proposition~\ref{PropSpecD},
and the infinite line $L$.
\end{prop}

I am grateful to Jonathan Breuer, Chris Godsil, Rostislav Grigorchuk,
and Tatiana Nagnibeda,
for useful indications and remarks during the writing of this paper.

\section{\textbf{Spectral measures
and the Hahn--Hellinger Multiplicity Theorem}}
\label{SectionHH}

This section is a reminder on various notions of spectral measures
and on the theorem of the title,
which is due to E.\ Hellinger in 1907 and H.\ Hahn in 1912;
references to the original papers can be found in \cite[Section X.6, Page 928]{DuSc--63}.
All Hilbert spaces which appear here are complex,
and separable whenever needed.
The scalar product of two vectors $\xi, \eta$ in a Hilbert space $\Hi$
is denoted by $\langle \xi \sca \eta \rangle$;
it is linear in $\xi$ and antilinear in $\eta$.
We use the following notation: $\N = \{0, 1, 2, \hdots, \}$
and $\overline{\N^*} = \{1, 2, \hdots, \infty \}$.

\subsection{Spectrum, spectral measures, and dominant vectors}
\label{SubsSp+Mu}

Let $\Hi$ be a Hilbert space,
$\Li (\Hi)$ the algebra of bounded linear operators on $\Hi$,
and $X \in \Li (\Hi)$.
The \textbf{spectrum} of $X$ is the set $\Sigma (X)$
of $\lambda \in \C$ such that
$\lambda {\rm Id} - X$ is not invertible in $\Li (\Hi)$.
It is a compact subset of $\C$, and a non-empty one unless $\Hi = \{0\}$.
Assume from now on that $X$ is \textbf{self-adjoint},
so that $\Sigma (X)$ is a compact subset of $\R$.
Denote by $\B_{\Sigma (X)}$
the $\sigma$-algebra of Borel subsets of $\Sigma (X)$.
By the spectral theorem, there exists a \textbf{projection-valued spectral measure}
$E_X \, \colon \B_{\Sigma (X)} \to {\rm Proj}(\Hi)$
such that $X = \int_{\Sigma (X)} x dE_X(x)$.
A vector $\xi \in \Hi$ determines
a \textbf{local spectral measure at $\xi$} on $\Sigma (X)$,
denoted by $\mu_\xi$, defined by
$\mu_\xi (B) = \langle E_X(B) \xi \sca \xi \rangle$
for all $B \in \B_{\Sigma (X)}$;
then $\langle X\xi \sca \xi \rangle = \int_{\Sigma(X)} x d\mu_\xi(x)$.
A vector $\xi$ is \textbf{dominant} for $X$
if $\mu_\eta$ is absolutely continuous with respect to $\mu_\xi$
for all $\eta \in \Hi$.
(``Dominant vector'' is the terminology of \cite[Page 306]{Sim4--15};
the terminology of \cite[Page 446]{BoSm--20} is
``vector of maximal type'', and that of \cite{Dixm--69}
is ``separating vector'' for the W$^*$-algebra generated by $X$.)
A \textbf{scalar-valued spectral measure for $X$}
is a measure on $\Sigma (X)$ of the form $\mu_\xi$, for $\xi$ dominant.
Two scalar-valued spectral measures for $X$ are equivalent,
i.e., are absolutely continuous with respect to each other.
A vector $\xi \in \Hi$ is \textbf{cyclic} for $X$
if the closed linear span of $\{X^n \xi\}_{n \in \N}$
is the whole of $\Hi$.
\par

We denote by $\mathcal B (\Sigma (X))$
the algebra of bounded Borel-measurable functions on $\Sigma (X)$.
For $f$ in this algebra, the operator $f(X)$
is defined by Borel functional calculus.

\begin{prop}[\textbf{existence and characterizations of dominant vectors
for self-adjoint operators}]
\label{PropDominant}
Let $X$ be a bounded self-adjoint operator on a separable Hilbert space $\Hi$.
Let $\mathcal B (\Sigma (X))$ and $E_X$ be as above.
\begin{enumerate}[label=(\arabic*)]
\item\label{1DEPropDominant}
There exist dominant vectors for $X$.
More precisely, for any $\eta \in \Hi$, there exists a dominant vector $\xi$ for $X$
such that $\eta$ is in the closed linear span of $\{ X^n \xi \}_{n \in \N}$.
\item\label{2DEPropDominant}
A vector $\xi \in \Hi$ is dominant for $X$ if and only if,
for any $f \in \mathcal B (\Sigma (X))$,
the equality $f(X) \xi = 0$ implies $f(X) = 0$.
\item\label{3DEPropDominant}
A vector $\xi \in \Hi$ is dominant for $X$ if and only if,
for any Borel subset $B$ of $\Sigma (X)$, the equality
$\mu_\xi(B) = 0$ is equivalent to the equality $E_X(B) = 0$.
\item\label{4DEPropDominant}
Cyclic vectors for $X$ are dominant vectors for $X$.
\item\label{5DEPropDominant}
If $X$ has at least one cyclic vector, dominant vectors for $X$
are cyclic vectors for $X$.
\end{enumerate}
Let $(\varepsilon_j)_{j \ge 1}$ be
an orthonormal basis of $\Hi$.
For $j \ge 1$, let $\mu_j$ denote the local spectral measure at $\varepsilon_j$.
\begin{enumerate}[label=(\arabic*)]
\addtocounter{enumi}{5}
\item\label{6DEPropDominant}
If $\xi \in \Hi$ is such that the local spectral measure $\mu_\xi$
dominates $\mu_j$ for all $j \ge 1$,
then $\xi$ is a dominant vector.
\end{enumerate}
\end{prop}

\begin{proof}[References for the proof]
For~\ref{1DEPropDominant}
and~\ref{2DEPropDominant},
see~\cite{Sim4--15}, Lemma 5.4.7 and Problem~3 of \S~5.4.
For~\ref{3DEPropDominant},
see~\cite{Conw--07}, Theorem IX.8.9.
For~\ref{4DEPropDominant}, let $\xi \in \Hi$
and $f \in \mathcal B (\Sigma (X))$ be such that $f(X) \xi = 0$;
then $f(X) X^n \xi = X^n f(X) \xi = 0$ for all $n \ge 0$,
hence $f(X) \eta = 0$ for all $\eta$
in the closed convex hull of $\{X^n \xi\}_{n \in \N}$;
if $\xi$ is cyclic then $f(X) \eta = 0$ for all $\eta \in \Hi$,
hence $f(X) = 0$, and therefore $\xi$ is dominant.
We leave the proofs of~\ref{5DEPropDominant}
and~\ref{6DEPropDominant} to the reader;
alternatively, see~\cite[Proposition 2.2 and Corollary 2.5]{BrHN}.
\end{proof}

The marked spectrum of a scalar multiple of a bounded self-adjoint operator
can easily be written in terms of the marked spectrum of the original operator.
For future reference, we make this precise in the following proposition,
which is an immediate consequence of the definitions.

\begin{prop}
\label{rescale}
Let $\Hi$ be a separable Hilbert space,
$X$ a bounded self-adjoint operator on $\Hi$,
and $\xi \in \Hi$.
Let $k > 0$ be a positive real number and let $Y = k X$.
Denote by $\mu^X_\xi$ the local spectral measure of $X$ at $\xi$
and by $\mu^Y_\xi$ the local spectral measure of $Y$ at~$\xi$.
Let $\mathopen[ m, M \mathclose]$ be the convex hull of the spectrum of $X$.
Assume that $\mu^X_\xi$ is of the form $\rho^X_\xi \lambda$,
where $\rho^X_\xi$ is a function in $L^1( \mathopen[ m, M \mathclose] , \lambda)$
with values in~$\R_+$
and where $\lambda$ denotes the Lebesgue measure
on $\mathopen[ m, M \mathclose]$.
Then:
\begin{enumerate}[label=(\arabic*)]
\item\label{1DErescale}
$\Vert Y \Vert = k \Vert X \Vert$,
\item\label{2DErescale}
$\Sigma (Y) = k \Sigma (X)$,
\item\label{3DErescale}
$\mu^Y_\xi = \rho^Y_\xi \lambda$
where $\rho^Y_\xi (x) = k^{-1} \rho^X_\xi (x/k)$
for all $x \in \mathopen[ km, kM \mathclose]$.
\end{enumerate}
\end{prop}

For our analysis of lattice graphs $L_d$ in Section~\ref{SectionRayLL},
we will need the following facts on local spectral measures
of some operators defined on tensor products.
Let $\Hi_1, \Hi_2$ be two separable Hilbert spaces.
For $j \in \{1, 2\}$,
let $X_j$ be a bounded self-adjoint operator on $\Hi_j$;
choose a vector $\xi_j$ in $\Hi_j$,
and let $\mu_j$ be the local spectral measure of $X_j$ at $\xi_j$;
we view $\mu_j$ as a finite measure on $\R$
with closed support contained in $\Sigma (X_j)$.
Let ${\rm Id}_j$ denote the identity operator on $\Hi_j$.
Let $\Hi$ be the Hilbert space tensor product $\Hi_1 \otimes \Hi_2$
and let $X \in \Li(\Hi)$ be the operator
$X_1 \otimes {\rm Id}_2 + {\rm Id}_1 \otimes X_2$.
It is well-known that the operator $X$ is bounded, self-adjoint,
of norm $\Vert X \Vert = \Vert X_1 \Vert + \Vert X_2 \Vert$,
and of spectrum
$$
\Sigma (X) = \{ z \in \R : z = x+ y
\hskip.2cm \text{for some} \hskip.2cm
x \in \Sigma (X_1)
\hskip.2cm \text{and} \hskip.2cm
y \in \Sigma (X_2)
\}
$$
(see \cite{BrPe--66} or \cite{Sche--69}).
Let $\xi = \xi_1 \otimes \xi_2 \in \Hi$
and let $\mu$ be the local spectral measure of $X$ at~$\xi$.

\begin{prop}
\label{localmeasuretensorproduct}
Let $X_1$ be a self-adjoint operator on $\Hi_1$
and $X_2$ a self-adjoint operator on $\Hi_2$;
let $\xi_1$, $\xi_2$, $\mu_1$, $\mu_2$,
$X = X_1 \otimes {\rm Id}_2 + {\rm Id}_1 \otimes X_2$,
$\xi = \xi_1 \otimes \xi_2$,
and $\mu$ be as above.
\par

Then $\mu$ is the convolution product $\mu_1 \ast \mu_2$.
\end{prop}

\begin{proof}
Recall that the convolution of two finite measure $\nu_1, \nu_2$ on $\R$
is the direct image of the measure $\nu_1 \otimes \nu_2$ on $\R^2$
by the map $\R^2 \to \R , \hskip.2cm (x,y) \mapsto x+y$.
We have
$$
\int_\R f(z) d (\nu_1 \ast \nu_2)(z) = \int_\R \int_\R f(x+y) d\nu_1(x) d\nu_2(y) 
$$
for any continuous function $f \, \colon \R \to \C$
which tends to zero at infinity;
when $\nu_1$ and $\nu_2$ are measures with compact support,
this holds more generally
for any continuous function $f \, \colon \R \to \C$.
See \cite[Chapter 7, Exercise 5]{Rudi--66}.
\par

For the next computation, observe that the operators
$X_1 \otimes {\rm Id}_2$ and ${\rm Id}_1 \otimes X_2$ commute,
and that
$\left( X_1 \otimes {\rm Id}_2 \right)^j \left( {\rm Id}_1 \otimes X_2 \right)^k
= X_1^j \otimes X_2^k$ for all $j, k \ge 0$.
For all $n \in \N$, we have
$$
\begin{aligned}
\int_{\Sigma (X)} & z^n d\mu(z)
= \langle ( X_1 \otimes {\rm Id}_2 + {\rm Id}_1 \otimes X_2 )^n (\xi_1 \otimes \xi_2)
\sca \xi_1 \otimes \xi_2 \rangle
\\
&= \left\langle 
\sum_{j=0}^n \binom{n}{j} (X_1^j \otimes X_2^{n-j}) (\xi_1 \otimes \xi_2) 
\hskip.2cm \bigg\vert \hskip.2cm
\xi_1 \otimes \xi_2 \right\rangle
\\
&= \sum_{j=0}^n \binom{n}{j} 
\langle X_1^j \xi_1 \sca \xi_1 \rangle
\langle X_2^{n-j} \xi_2 \sca \xi_2 \rangle ,
\end{aligned}
$$
hence
$$
\begin{aligned}
\int_{\Sigma (X)} & z^n d\mu(z)
= \sum_{j=0}^n \binom{n}{j} 
\int_{\Sigma (X_1)} x^j d\mu_1(x) \int_{\Sigma (X_2)} y^{n-j} d\mu_2 (y)
\\
&= \int_{\Sigma (X_1)} \int_{\Sigma (X_2)}
\left( \sum_{j=0}^n \binom{n}{j} x^j y^{n-j} \right) d\mu_1(x) d\mu_2(y)
\\
&= \int_{\Sigma (X_1)} \int_{\Sigma (X_2)}
(x+y)^n d\mu_1(x) d\mu_2(y)
= \int_{\Sigma (X)} z^n d(\mu_1 \ast \mu_2) (z) .
\end{aligned}
$$
This shows that the moments of $\mu$
are the same as the moments of $\mu_1 \ast \mu_2$.
Since $\mu$ and $\mu_1 \ast \mu_2$ are measures with compact support,
it follows that $\mu = \mu_1 \ast \mu_2$.
\end{proof}

\begin{rem}
\label{remconvol}
For each positive integer $d$, there is a similar fact which holds
for operators of the form
$$
X = X_1 \otimes {\rm Id_2} \otimes \cdots \otimes {\rm Id_d}
+ {\rm Id_1} \otimes X_2 \otimes \cdots \otimes {\rm Id_d}
+ \cdots + {\rm Id_1} \otimes {\rm Id_2} \otimes \cdots \otimes X_d
$$
which have local spectral measures of the form
$\mu = \mu_1 \ast \mu_2 \ast \cdots \ast \mu_d$.
\end{rem}

\subsection{Multiplication operators}
\label{SubsMultOp}

We recall successively the definition
of the Hilbert space $L^2(\Sigma, \mu, \mathfrak m)$,
some facts on functions $\varphi \in L^\infty(\Sigma, \mu)$,
and on multiplications operators $M_{\Sigma, \mu, \mathfrak m, \varphi}$.
\par

Let $\Sigma$ be a non-empty metrizable compact space.
Let $\B_\Sigma$ the $\sigma$-algebra of Borel subsets of $\Sigma$,
and $\mu$ a finite positive measure on $(\Sigma, \B_\Sigma)$.
Let $\mathfrak m \, \colon \Sigma \to \overline{\N^*}$ be a measurable function.
Denote by $\ell^2_\infty$ the Hilbert space
of square summable sequences $(z_j)_{j \ge 1}$ of complex numbers
and, for each $n \ge 1$, by $\ell^2_n$ the subspace
of sequences such that $z_j = 0$ for all $j \ge n+1$.
Let $L^2(\Sigma, \mu, \mathfrak m)$
be the separable Hilbert space of measurable functions $\xi \, \colon \Sigma \to \ell^2_\infty$
such that $\xi (x) \in \ell^2_{\mathfrak m (x)}$ for all $x \in \Sigma$
and $\int_\Sigma \Vert \xi (x) \Vert^2_{\ell^2_\infty} d\mu(x) < \infty$.
In more sophisticated terms, $L^2(\Sigma, \mu, \mathfrak m)$ is the Hilbert space
of square summable vector fields
of the $\mu$-measurable field of Hilbert spaces $(\Hi_x)_{x \in \Sigma}$,
where $\Hi_x = \Ki_{\mathfrak m(x)}$ for all $x \in \Sigma$.
The space $L^2(\Sigma, \mu, \mathfrak m)$ can also be seen as a Hilbert direct sum
$\bigoplus_{n \in \overline{\N^*}} L^2 (\Sigma_n, \mu_n, \Ki_n)$,
where $\Sigma_n = \mathfrak m^{-1}(n)$,
the measure $\mu_n$ is defined by $\mu_n(B) = \mu(B \cap \Sigma_n)$
for all Borel sets $B \in \B_\Sigma$,
and $L^2 (\Sigma_n, \mu_n, \Ki_n)$ is the Hilbert space
of square-summable $\Ki_n$-valued functions on $(\Sigma_n, \mu_n)$.
Note that $\Sigma_n = \emptyset$ when $\mathfrak m (x) \ne n$
for all $x \in \Sigma$, 
and more generally that $\mu_n = 0$
and $L^2 (\Sigma_n, \mu_n, \Ki_n) = \{ 0 \}$
when $\mathfrak m (x) \ne n$
for $\mu$-almost all $x \in \Sigma$.
Note also that $\mu_n$ can be seen
either as a measure on $\Sigma_n$,
or as a measure on $\Sigma$
such that $\mu_n(\Sigma \smallsetminus \Sigma_n) = 0$;
in the latter case, the measures $\mu_n$~'s are pairwise singular with each other.
\par

Let $\varphi \, \colon \Sigma \to \R$ be a measurable
complex-valued function on $\Sigma$.
The \textbf{essential supremum} of $\varphi$
is the infimum $\Vert \varphi \Vert_\infty$ of the numbers $c \ge 0$
such that $\mu \left( \{ x \in \Sigma : \vert \varphi(x) \vert > c \} \right) = 0$.
We assume from now on that $\varphi$ is \textbf{essentially bounded},
i.e., that $\Vert \varphi \Vert_\infty < \infty$.
The \textbf{essential range} of $\varphi$
is the set $R_\varphi$ of complex numbers~$z$ such that
$\mu ( \{ x \in \Sigma : \vert \varphi(x) - z \vert < \varepsilon \} ) > 0$
for all $\varepsilon > 0$;
we have $\Vert \varphi \Vert_\infty = \sup \{ \vert z \vert : z \in R_\varphi \}$.
In other words,
$R_\varphi$ is the closed support of the measure $\varphi_*(\mu)$ on $\C$,
the push forward of $\mu$ by $\varphi$,
and therefore $R_\varphi$ is a closed subset of $\C$,
indeed a compact subset of $\C$ since $\varphi$ is essentially bounded.
Below, $\Vert \varphi \Vert_\infty$ and $R_\varphi$ will be
the norm and the spectrum of a multiplication operator.
\par

For $z \in \C$ and $\varepsilon > 0$, let $D_\varepsilon(z)$
denote the closed disc $\{ w \in \C : \vert w - z \vert \le \varepsilon \}$.
Note that $\mu ( \varphi^{-1} (D_\varepsilon(z))) > 0$ for all $\varepsilon > 0$
when $z \in R_\varphi$.
For $z \in R_\varphi$, the \textbf{essential pre-image} $\varphi^{-1}_\mu (z)$
is defined as the set of those $x \in \Sigma$ for which,
for every neighborhood $V$ of $x$ in $\Sigma$, we have
$$
\liminf_{\varepsilon \to 0} \frac
{ \mu \left( V \cap \varphi^{-1} (D_\varepsilon(z) ) \right) }
{ \mu \left( \varphi^{-1} (D_\varepsilon(z) ) \right) }
> 0 .
$$
For $z \in \C \smallsetminus R_\varphi$, set $\varphi^{-1}_\mu (z) = \emptyset$.
When $\varphi$ is continuous, $\varphi^{-1}_\mu (z)$
is contained in $\varphi^{-1}(z)$
\cite[Theorem 6]{AbKr--73};
equality need not hold 
\cite[Pages 853--854]{AbKr--73}.
Below, the cardinalities of the essential pre-images of $\varphi$
will be the values of the spectral multiplicity function of a multiplication operator.
\par

Let $\varphi, \varphi' \, \colon \Sigma \to \C$ be two measurable functions
which are equal $\mu$-almost every where;
then the norms $\Vert \varphi \Vert_\infty$, $\Vert \varphi' \Vert_\infty$ are equal,
$\varphi, \varphi'$ have the same essential range,
and $\varphi, \varphi'$ have the same essential pre-images.
From now on, we consider such functions as being equal,
and write (abusively) ``function''
for ``equivalence class of functions modulo equality $\mu$-almost everywhere''.
The space $L^\infty (X, \mu)$
of essentially bounded complex-valued functions on $(\Sigma, \mu)$
is a Banach space for the norm $\Vert \cdot \Vert_\infty$.
It is the dual of $L^1 (X, \mu)$,
hence it can be considered with both its norm topology
and its w$^*$-topology (see for example \cite[Theorem 1.45]{Doug--72}).
\par

Suppose that $\Sigma$ is a nonempty compact subset of the real line.
Denote by $\mathcal C (\Sigma)$ the algebra of continuous functions on $\Sigma$,
with the $\sup$-norm,
and by $\mathcal P (\Sigma)$ the subalgebra of functions
which are restrictions to $\Sigma$ of polynomial functions on $\R$.
Then $\mathcal P (\Sigma)$ is dense in $\mathcal C (\Sigma)$,
by the Stone--Weierstrass theorem,
and the natural image of $\mathcal C (\Sigma)$ in $L^\infty (\Sigma, \mu)$
is w$^*$-dense, see \cite[Corollary 4.53]{Doug--72}.
It follows that $\mathcal P (\Sigma)$ is w$^*$-dense in $L^\infty (\Sigma, \mu)$.

\begin{defn}
\label{DefMO}
Let $\Sigma, \mu, \mathfrak m$ and $\varphi$ be as above.
The \textbf{multiplication operator} $M_{\Sigma, \mu, \mathfrak m, \varphi}$
is the operator defined on the space $L^2(\Sigma, \mu, \mathfrak m)$ by
$$
(M_{\Sigma, \mu, \mathfrak m, \varphi} \xi) (x) = \varphi(x) \xi (x)
\hskip.5cm \text{for all} \hskip.2cm
\xi \in L^2(\Sigma, \mu, \mathfrak m)
\hskip.2cm \text{and} \hskip.2cm
x \in \Sigma .
$$
When $\mathfrak m$ is the constant function of value $1$,
we write $M_{\Sigma, \mu, \varphi}$
instead of $M_{\Sigma, \mu, \mathfrak m, \varphi}$.
\par

A \textbf{straight multiplication operator} $M_{\Sigma, \mu, \mathfrak m}$
is an operator of this type in the particular case
of a compact subset $\Sigma$ of the real line
and of the function $\varphi$ given by the inclusion $\Sigma \subset \R$,
so that $(M_{\Sigma, \mu, \mathfrak m} ) (x) = x \xi (x)$
for all $\xi \in L^2(\Sigma, \mu, \mathfrak m)$
and $x \in \Sigma$.
\end{defn}

\begin{prop}
\label{PropMO}
Let $\Sigma$, $\mu$,
$\mathfrak m \, \colon \Sigma \to \overline{\N^*}$,
$L^2(\Sigma, \mu, \mathfrak m)$,
$\varphi \in L^\infty(\Sigma, \mu)$ be as above,
and $M_{\Sigma, \mu, \mathfrak m, \varphi}$
the corresponding multiplication operator,
as in Definition~\ref{DefMO}.
Suppose now that $\varphi$ is a real-valued function.
\begin{enumerate}[label=(\arabic*)]
\item\label{1DEPropMO}
$M_{\Sigma, \mu, \mathfrak m, \varphi}$ is a bounded self-adjoint operator
with norm $\Vert M_{\Sigma, \mu, \mathfrak m, \varphi} \Vert
= \Vert \varphi \Vert_\infty$.
\item\label{2DEPropMO}
The spectrum of $M_{\Sigma, \mu, \mathfrak m, \varphi}$ is the essential range $R_\varphi$ of $\varphi$,
and $\lambda \in \R$ is an eigenvalue of $M_{\Sigma, \mu, \mathfrak m, \varphi}$
if and only if $\mu (\{ x \in \Sigma : \varphi(x) = \lambda \}) > 0$.
\item\label{3DEPropMO}
The spectral measure $E_{M_{\Sigma, \mu, \mathfrak m, \varphi}}$ is given by
$E_{M_{\Sigma, \mu, \mathfrak m, \varphi}} (B) = M_{\Sigma, \mu, \mathfrak m, \chi_{ \varphi^{-1}(B) } }$
for any Borel subset~$B$ of $R_\varphi$,
where $\chi_{ \varphi^{-1} (B) }$ stands for the characteristic function
of the inverse image of $B$ by $\varphi$.
\item\label{4DEPropMO}
The measure $\mu$
is a scalar-valued spectral measure for~$M_{\Sigma, \mu, \mathfrak m, \varphi}$.
\end{enumerate}
\par

Suppose that, in particular, $\Sigma \subset \R$
and that $\varphi$ is given by the inclusion $\Sigma \subset \R$;
let $M_{\Sigma, \mu, \mathfrak m}$ be
the corresponding straight multiplication operator,
as in Definition~\ref{DefMO}.
Let $\Sigma_\mu$ denote the closed support of $\mu$.
\begin{enumerate}[label=(\arabic*)]
\addtocounter{enumi}{4}
\item\label{5DEPropMO}
$\Vert M_{\Sigma, \mu, \mathfrak m} \Vert
= \sup \{ \vert x \vert : x \in \Sigma_\mu \}$.
\item\label{6DEPropMO}
The spectrum of $M_{\Sigma, \mu, \mathfrak m}$ is $\Sigma_\mu$.
\item\label{7DEPropMO}
$E_{M_{\Sigma, \mu, \mathfrak m}} (B) = M_{\Sigma, \mu, \mathfrak m, \chi_B}$
for any Borel subset $B$ of $\Sigma_\mu$.
\item\label{8DEPropMO}
$\mu$ is a scalar-valued spectral measure for $M_{\Sigma, \mu, \mathfrak m}$. 
\end{enumerate}
\par

Suppose moreover that $\mathfrak m = \mathbf 1_\Sigma$
is the constant function of value $1$,
so that the operator $M = M_{\Sigma, \mu, \mathbf 1_\Sigma}$
acts on $L^2(\Sigma, \mu)$.
\begin{enumerate}[label=(\arabic*)]
\addtocounter{enumi}{8}
\item\label{9DEPropMO}
For $\xi \in L^2(\Sigma, \mu)$,
the following conditions are equivalent:
\begin{enumerate}[label=(\roman*)]
\item
$\xi$ is cyclic for $M$,
\item
$\xi$ is dominant for $M$,
\item
$\mu (\{ x \in \Sigma : \xi(x) = 0 \}) = 0$.
\end{enumerate}
\item\label{10DEPropMO}
In particular, the function on $\Sigma$ of constant value $1$
is a cyclic vector for~$M$.
\end{enumerate}
\end{prop}

\begin{proof}[On the proof]
Let $\mathfrak m_\mu$ denote
the restriction of the function $\mathfrak m$ to $\Sigma_\mu$.
The spaces $L^2(\Sigma, \mu, \mathfrak m)$
and $L^2(\Sigma_\mu, \mu, \mathfrak m_\mu)$
are canonically isomorphic,
and $M$ can be seen as an operator on $L^2(\Sigma_\mu, \mu, \mathfrak m_\mu)$.
It follows that we can assume without loss of generality
that $\Sigma = \Sigma_\mu$,
namely that the closed support of $\mu$ is the whole of $\Sigma$.
\par

The arguments to prove Claims~\ref{1DEPropMO} to~\ref{4DEPropMO} are standard;
see for example Sections 4.20 to 4.28 in~\cite{Doug--72},
or any of \cite{AbKr--73, Abra--78, Krie--86}.
\par

Let $\xi \in L^2 (\Sigma, \mu)$.
Suppose first that the condition
$\mu (\{ x \in \Sigma : \xi(x) = 0 \}) = 0$ of~\ref{9DEPropMO}~(iii) is satisfied.
Let $\eta \in L^2 (\Sigma, \mu)$
be orthogonal to $M^n \xi$ for all $n \in \N$;
we are going to show that $\eta = 0$.
Note that the product $\xi \overline{\eta}$
is in the weak$^*$ dual $L^1 (\Sigma, \mu)$ of $L^\infty (\Sigma, \mu)$, 
because $\xi$ and $\eta$ are in $L^2 (\Sigma, \mu)$.
Since
$\langle M^n \xi \sca \eta \rangle
= \int_\Sigma x^n \xi(x) \overline{ \eta(x) } \mu(x) = 0$
for all $n \in \N$, we have
$$
\int_\Sigma f(x) \xi(x) \overline{ \eta(x) } d\mu(x) = 0
$$
for all $f \in \mathcal P (\Sigma)$,
and therefore also for all $f \in L^\infty (\Sigma, \mu)$
because $\mathcal P (\Sigma)$ is w$^*$-dense in $L^\infty (\Sigma, \mu)$.
This implies that $\xi \overline{ \eta } = 0$ in $L^1 (\Sigma, \mu)$,
hence that $\xi (x) \eta (x) = 0$ for $\mu$-almost all $x \in \Sigma$, 
hence by hypothesis on $\xi$ that $\eta (x) = 0$ for $\mu$-almost all $x \in \Sigma$,
hence that $\eta = 0$.
It follows that $\xi$ is cyclic for $M$.
\par

This shows~\ref{10DEPropMO}
because the condition of~\ref{9DEPropMO}~(iii) is clearly satisfied.
for $\xi$ the constant function of value $1$.
Moreover, a vector in $L^2(\Sigma, \mu)$ is cyclic for $M$
if and only if it is dominant for $M$,
by Proposition~\ref{PropDominant}.
\par

Suppose now on the contrary $\xi \in L^2 (\Sigma, \mu)$
such that $\mu (\{ x \in \Sigma : \xi(x) = 0 \}) > 0$.
Define a Borel function $\chi \, \colon \Sigma \to \C$
by $\chi(x) = 1$ when $x$ is such that $\xi(x) \ne 0$
and $\chi(x) = 0$ otherwise.
Then $\chi(M) \ne 0$ and $\chi(M) \xi = 0$.
It follows that $\xi$ is not dominant for~$M$.
\par

This concludes the proof of~\ref{9DEPropMO}. 
\end{proof}

An operator $X_1$ on a Hilbert space $\Hi_1$
and an operator $X_2$ on a Hilbert space $\Hi_2$
are \textbf{unitarily equivalent}
if there exists a unitary operator (=~a surjective isometry)
$U \, \colon \Hi_1 \to \Hi_2$ such that $X_2 = UX_1U^*$.
\par

If $X_1 \in \Li(\Hi_1)$ and $X_2 \in \Li(\Hi_2)$ are two self-adjoint operators
which are unitarily equivalent, their spectra coincide, $\Sigma(X_1) = \Sigma(X_2)$,
and their scalar-valued spectral measures are the same.

\begin{exa}[\textbf{unitarily equivalent pairs of multiplication operators}]
\label{exuniteq}
Let $\mathopen[ a_1, b_1 \mathclose]$, $\mathopen[ a_2, b_2 \mathclose]$
be two intervals of the real line,
with $-\infty < a_1 < b_1 < \infty$ and $-\infty < a_2 < b_2 < \infty$.
We consider the Hilbert spaces $L^2( \mathopen[ a_1, b_1 \mathclose], \lambda)$
and $L^2( \mathopen[ a_2, b_2 \mathclose], \lambda)$,
where $\lambda$ is the Lebesgue measure.
Let
$$
\varphi_2 \, \colon \mathopen[ a_2, b_2 \mathclose]
\overset{\approx}{\longrightarrow} \mathopen[ a_1, b_1 \mathclose]
$$
be a function of class $\mathcal C^1$, injective,
mapping $\mathopen[ a_2, b_2 \mathclose]$
onto $\mathopen[ a_1, b_1 \mathclose]$,
and such that $\vert \varphi_2' (x) \vert > 0$
for all $x \in \mathopen]a_2, b_2\mathclose[$.
Define an operator
$M_1 = M_{ \mathopen[ a_1, b_1 \mathclose], \lambda, \mathbf 1}$
on $L^2( \mathopen[ a_1, b_1 \mathclose], \lambda)$
by
$$
(M_1 \xi_1) (x) = x \xi_1(x)
\hskip.5cm \text{for all} \hskip.2cm
\xi_1 \in L^2( \mathopen[ a_1, b_1 \mathclose])
\hskip.2cm \text{and} \hskip.2cm
x \in [a_1, b_1]
$$
and an operator $M_2 = M_{ \mathopen[ a_2, b_2 \mathclose], \lambda, \mathbf 1, \varphi_2 }$
on $L^2( \mathopen[ a_2, b_2 \mathclose], \lambda)$
by
$$
(M_2 \xi_2) (x) = \varphi_2(x) \xi_2(x)
\hskip.5cm \text{for all} \hskip.2cm
\xi_2 \in L^2( \mathopen[ a_2, b_2 \mathclose])
\hskip.2cm \text{and} \hskip.2cm
x \in [a_2, b_2] .
$$
\par
 
Then $M_1$ and $M_2$ are unitarily equivalent.
\end{exa}

\begin{proof}
Let $U \, \colon L^2( \mathopen[ a_1, b_1 \mathclose], \lambda)
\to L^2( \mathopen[ a_2, b_2 \mathclose], \lambda)$
be the operator defined by
$$
(U \xi_1) (x) = \sqrt{ \vert \varphi_2'(x) \vert } \hskip.1cm \xi_1( \varphi_2 (x) )
\hskip.5cm \text{for all} \hskip.2cm
\xi_1 \in L^2( \mathopen[ a_1, b_1 \mathclose], \lambda)
\hskip.2cm \text{and} \hskip.2cm
x \in [a_2, b_2] .
$$
Then $U$ is unitary.
Indeed, for
$\xi_1 \in L^2(\mathopen[ a_1, b_1 \mathclose], \lambda)$
and $\xi_2 \in L^2(\mathopen[ a_2, b_2 \mathclose], \lambda)$,
we have
$$
\begin{aligned}
\Vert U \xi_1 \Vert^2 &= \int_{a_2}^{b_2} \vert (U \xi_1) (x) \vert^2 dx
= \int_{a_2}^{b_2} \vert \xi_1 (\varphi_2(x)) \vert^2 \vert \varphi_2'(x) \vert dx
\\
&= \int_{a_1}^{b_1} \vert \xi_1 (y) \vert^2 dy = \Vert \xi_1 \Vert^2 ,
\end{aligned}
$$
and similarly $\Vert U^{-1} \xi_2 \Vert^2 = \Vert \xi_2 \Vert^2$.
\par

For
$\xi_1 \in L^2(\mathopen[ a_1, b_1 \mathclose], \lambda)$,
we have
$$
\begin{aligned}
(M_2 U \xi_1)(x)
&= \varphi_2 (x) \left( \sqrt{ \vert \varphi_2'(x) \vert } \hskip.1cm \xi_1 \big( \varphi_2(x) \big) \right)
= \sqrt{ \vert \varphi_2'(x) \vert } \Big( \varphi_2(x) \xi_1( \varphi_2(x) ) \Big)
\\
&= \sqrt{ \vert \varphi_2'(x) \vert } \hskip.1cm (M_1 \xi_1) (\varphi_2(x))
= (U M_1 \xi_1) (x) .
\end{aligned}
$$
It follows that $M_2 U = U M_1$, and this ends the proof.
\end{proof}

Here are two particular cases; this will be useful in the proof of Proposition~\ref{infline}.

\begin{exa}
\label{ExCosEqt}
(1)
Let $\mathopen[ a_1, b_1 \mathclose] = \mathopen[ a_2, b_2 \mathclose]
= \mathopen[ 0, 1 \mathclose]$
and $\varphi_2 (x) = x^\alpha$ for some $\alpha \in \R, \alpha > 0$.
The operator $M_1$ of multiplication by $x$
and the operator $M_\alpha$ of multiplication by $x^\alpha$
on $L^2(\mathopen[ 0, 1 \mathopen], \lambda)$ are unitarily equivalent.
The unitary operator $U$ on $L^2(\mathopen[ 0, 1 \mathopen], \lambda)$ is given by
$(U\xi) (x) = \sqrt{ \alpha x^{\alpha-1} } \xi(x^\alpha)$,
and $M_\alpha U = U M_1$.
\par

(2)
Let $\mathopen[ a_1, b_1 \mathclose] = \mathopen[ -2 , 2 \mathclose]$,
$\mathopen[ a_2, b_2 \mathclose] = \mathopen[ 0 , \pi \mathclose]$,
and $\varphi_2 (x) = 2 \cos (x)$.
The operator $M_1$ of multiplication by $x$
on $L^2( \mathopen[ -2, 2 \mathclose], \lambda)$
and the operator $M_{2 \cos}$ of multiplication by $2 \cos (x)$
on $L^2( \mathopen[ 0, \pi \mathclose], \lambda )$
are unitarily equivalent.
Similarly, the operator $M_1$ on $L^2( \mathopen[ -2, 2 \mathclose], \lambda)$
and the operator of multiplication by $2 \cos (x)$ on $L^2( \mathopen[ \pi, 2 \pi], \lambda)$
are unitarily equivalent.
\end{exa}

\subsection{The Hahn--Hellinger Multiplicity Theorem, and spectral multiplicity functions}
\label{SubsHH}

The following Theorem \ref{mainthHH} is the keystone of Hahn--Hellinger theory.

\begin{thm}
\label{mainthHH}
Any self-adjoint operator $X$ on a separable Hilbert space $\Hi$
is unitarily equivalent to
the straight multiplication operator $M_{\Sigma, \mu, \mathfrak m}$
of Definition~\ref{DefMO}
for the spectrum $\Sigma = \Sigma (X)$ of $X$,
a scalar-valued spectral measure $\mu$ for $X$,
and a measurable function $\mathfrak m \, \colon \Sigma \to \{1, 2, \hdots, \infty \}$.
\par

Moreover, if $\mu'$ is a measure on $\Sigma$
and $\mathfrak m' \, \colon \Sigma \to \{1, 2, \hdots, \infty \}$ a measurable function,
then $X$ is unitarily equivalent
to the straight multiplication operator $M_{\Sigma, \mu', \mathfrak m'}$
if and only if the measures $\mu, \mu'$ are equivalent,
and the functions $\mathfrak m, \mathfrak m'$ are equal $\mu$-almost everywhere.
\end{thm}

For a sample of other formulations of the theorem and for proofs, see
\cite[Theorem X.5.10]{DuSc--63},
\cite[Chap.~II, \S~6]{Dixm--69},
\cite[Section 2.2]{Arve--76},
\cite{Krie--86},
\cite[Theorem 10.16 and Theorem 10.20]{Conw--07},
\cite[Section~5.4]{Sim4--15}, and
\newline
\cite[Theorem 10.4.6]{BoSm--20}.

\begin{defn}
\label{DefMultF}
Let $\Hi$, $X$, $\Sigma$, $\mu$ and $\mathfrak m$ be as in the previous theorem.
The function $\mathfrak m$ is the \textbf{spectral multiplicity function} of $X$.
The operator $X$ is of \textbf{finite multiplicity}
if there exists a finite constant $N$ such that
$\mathfrak m (x) \le N$ for $\mu$-almost all $x \in \Sigma$.
The operator~$X$ is \textbf{multiplicity-free}, or simple,
if $\mathfrak m (x) = 1$ for $\mu$-almost all $x \in \Sigma$,
equivalently if it is unitarily equivalent
to the operator of multiplication by $x$
on the Hilbert space $L^2(\Sigma, \mu)$,
where $\mu$ is a scalar-valued spectral measure on the spectrum $\Sigma$ of $X$.
The operator $X$ is of \textbf{uniform multiplicity $n \in \overline{\N^*}$}
if $\mathfrak m(x) = n$ for $\mu$-almost all $x \in \Sigma (X)$,
equivalently if $X$ is unitarily equivalent to a direct sum
$X_1 \oplus \cdots \oplus X_n$ of
pairwise unitarily equivalent multiplicity-free self-adjoint operators $X_1, \hdots, X_n$.
\end{defn}

\begin{cor}[\textbf{reformulation of part of Theorem~\ref{mainthHH}}]
\label{CorRefUnitEq}
Let $X_1, X_2$ be two self-adjoint operators
on two Hilbert spaces $\Hi_1, \Hi_2$.
Suppose that $X_1$ and $ X_2$ have same spectrum,
equivalent scalar-valued spectral measures,
and spectral multiplicity functions which are equal almost everywhere;
in other words, suppose that $X_1$ and $X_2$ have the same marked spectrum.
\par

Then $X_1$ and $X_2$ are unitarily equivalent.
\end{cor}

\begin{prop}
\label{PropMultiplicityFree}
For a self-adjoint operator $X$ on a separable Hilbert space $\Hi$,
the following properties are equivalent:
\begin{enumerate}[label=(\roman*)]
\item\label{1DEPropMultiplicityFree}
$X$ is multiplicity-free,
\item\label{2DEPropMultiplicityFree}
$X$~has a cyclic vector.
\end{enumerate}
\end{prop}

\begin{proof}[Reference for a proof, and comments]
Suppose that $X$ satisfies Condition~\ref{1DEPropMultiplicityFree}.
Let $M_{\Sigma, \mu, \mathfrak m} = M_{\Sigma, \mu, \mathbf 1_\Sigma}$
be as in Theorem~\ref{mainthHH}.
Then $M_{\Sigma, \mu, \mathbf 1_\Sigma}$
has a cyclic vector by~\ref{10DEPropMO} of Proposition~\ref{PropMO},
hence $X$ has a cyclic vector.
\par

The converse implication~\ref{2DEPropMultiplicityFree}
$\Rightarrow$ \ref{1DEPropMultiplicityFree}
can be seen as one form of the spectral theorem;
we refer to \cite[Theorem 5.1.7]{Sim4--15}.
\end{proof}

The next proposition is a complement to Proposition~\ref{PropMO}
for the spectral multiplicity function, in the simple case
of $\mathfrak m = \mathbf 1_\Sigma$.
We use it below in the proof of Proposition~\ref{lattices}.
For the proof, we refer to~\cite[Theorem 5]{Abra--78}.

\begin{prop}
\label{PropMOmult}
Let $\Sigma$ be a non-empty metrizable compact space
and $\mu$ a finite positive measure on $\Sigma$;
assume that the closed support of $\mu$ is the whole of~$\Sigma$.
Let $\varphi$ be a \emph{continuous} real-valued function on $\Sigma$,
viewed as $\varphi \in L^\infty(\Sigma, \mu)$.
Let $M = M_{\Sigma, \mu, \mathbf 1_\Sigma, \varphi}$
be the multiplication operator by $\varphi$
on $L^2(\Sigma, \mu)$, as in Definition~\ref{DefMO};
recall from Proposition~\ref{PropMO} that the spectrum of $M$
is the essential range $R_\varphi$.
\par

Then the spectral multiplicity function
$\mathfrak m$ for $M_{\Sigma, \mu, \mathbf 1_\Sigma, \varphi}$
satisfies
$$
\mathfrak m (x) = \sharp \left( \varphi_\mu^{-1} (x) \right)
$$
for $\mu$-almost all $x \in \Sigma (M) = R_\varphi$.
\end{prop}

For infinite connected graphs,
spectral multiplicity functions of adjacency operators
have not been much studied. 
It would be interesting (at least for us!) to understand which of these graphs
have multiplicity-free adjacency operators.
\par

In contrast, many results have been shown
concerning finite graphs and adjacency operators with simple eigenvalues;
we quote a few.
\par

All eigenvalues are simple for finite paths,
for all trees with at most $10$ vertices
(see the tables of~\cite{CvDS--80}),
for all finite connected graphs $G$ with $\Vert A_G \Vert \le 2$.
\par

If $G$ is a finite graph such that all eigenvalues of $A_G$ are simple,
any automorphism of $G$ is of order $2$;
more precisely, the automorphism group of $G$
is an elementary abelian $2$-group
(\cite{Mows--69}, see also \cite[Corollary 1.6.1]{BrHa--12}).
\par

Let $G = (V, E)$ be a finite graph with $n = \vert V \vert$ vertices
and $A_G$ its adjacency matrix.
We denote by $\mathbf 1_V$ the vector in $\ell^2(V)$
defined by $\mathbf 1_V (v) = 1$ for all $v \in V$.
Say $G$ is \textbf{controllable} if $\mathbf 1_V$ is a cyclic vector for $A_G$.
It is conjectured in \cite{Gods--12} and proved in \cite{ORTo--16} that
\emph{almost all finite graphs are controllable},
and therefore multiplicity-free.
This is made precise as follows.
Consider a positive integer $n$ and a probability $p \in \mathopen] 0, 1 \mathclose[$.
Let $\mathcal{G} (n, p)$ be the set of all graphs with vertex set $\{1, 2, \hdots, n\}$
having $\lfloor p \binom{n}{2} \rfloor$ edges.
Let $\mathcal{MFG} (n, p)$ be the subset of $\mathcal{G} (n, p)$
of multiplicity-free graphs.
We denote by $\sharp S$ the cardinality of a set $S$.
Then
$$
\lim_{n \to \infty} \sharp \mathcal{MFG} (n, p) / \sharp \mathcal{G} (n, p) 
= 1 .
$$

\section{\textbf{The infinite ray, the infinite line, and the lattices}}
\label{SectionRayLL}

Let again $G = (V, E)$ be a graph,
$(\delta_v)_{v \in V}$ the standard orthonormal basis
of the Hilbert space $\ell^2(V)$,
and $A_G$ its adjacency operator on $\ell^2(V)$.
A vertex $v \in V$ is \textbf{dominant}
if the vector $\delta_v$ is dominant for $A_G$,
and $v$ is \textbf{cyclic} if the vector $\delta_v$ is cyclic for $A_G$.
The \textbf{vertex spectral measure at $v \in V$}
is the local spectral measure at $\delta_v$
on the spectrum $\Sigma (A_G)$ of $A_G$.

\vskip.2cm

The \textbf{infinite ray} is the graph $R$ with vertex set $\N = \{0, 1, 2, 3, \hdots\}$
and edge set $E = \{ \{j, j+1\} : j \in \N \}$.
The adjacency operator $A_R$ of $R$ is defined by
$$
(A_R \xi)(u)
= \xi(u-1) + \xi(u+1)
\hskip.5cm \text{for all} \hskip.2cm
\xi \in \ell^2(\N) 
\hskip.2cm \text{and} \hskip.2cm
u \in \N ,
$$
where $\xi(-1)$ should be read as $0$.
With respect to the standard basis $\left( \delta_n \right)_{n \in \N}$
of the Hilbert space $\ell^2(\N)$~
the adjacency operator $A_R$ is the \textbf{free Jacobi matrix}~$J$:
$$
A_R = J
= \begin{pmatrix}
0 & 1 & 0 & 0  & \cdots
\\
1 & 0 & 1 & 0  & \cdots
\\
0 & 1 & 0 & 1  & \cdots
\\
0 & 0 & 1 & 0  & \cdots
\\
\vdots & \vdots & \vdots & \vdots & \ddots
\end{pmatrix} 
$$
($J_{m,n} = 1$ if $\vert m-n \vert = 1$
and $J_{m,n} = 0$ otherwise).
The following proposition collects standard results on $J$.

\begin{prop}
\label{infray}
Let $R$ be the infinite ray and let $A_R = J$ be its adjacency operator.
\begin{enumerate}[label=(\arabic*)]
\item\label{1DEinfray}
The norm of $A_R$ is $2$.
\item\label{2DEinfray}
The spectrum of $A_R$ is
$\mathopen[ -2, 2 \mathclose]$,
and $A_R$ does not have any eigenvalue.
\item\label{3DEinfray}
The vertex spectral measure of $A_R$ at $0$ is given by
$d \mu (x) = \frac{1}{2\pi} \sqrt{ 4 - x^2} dx$
for $x \in \mathopen[ -2, 2 \mathclose]$;
it is a scalar-valued spectral measure for $A_R$.
\item\label{4DEinfray}
The vertex $0$ is cyclic in $R$
and the operator $A_R$ is multiplicity-free.
\end{enumerate}
\end{prop}

(It is known that all vertices of $R$ are cyclic \cite[Proposition 7.1]{BrHN}.)

\begin{proof}
The strategy of the proof is to view $J$
as the matrix of an operator of multiplication by $x$
on a Hilbert space of functions on $\mathopen[ -2, 2 \mathclose]$
with respect to an appropriate basis of orthogonal polynomials.
For some background on orthogonal polynomials
and their relations with Jacobi matrices,
see \cite[Section 4.1]{Sim4--15}.
\par

Consider the sequence $\left( P_n \right)_{n=0}^\infty$ of functions
defined on the interval $\mathopen[ -2, 2 \mathclose]$ of the real line by
$$
P_n( 2 \cos \theta ) = \frac{ \sin ( (n+1)\theta ) }{ \sin (\theta) } 
$$
for $\theta \in \mathopen[ 0, \pi \mathclose]$.
Note that $P_0 (x) = 1$, $P_1(x) = x$, $P_2(x) = x^2-1$,
for all $x \in \mathopen[ -2, 2 \mathclose]$.
Define $P_{-1}$ to be the zero function.
From the trigonometric formula
$$
2 \cos (\theta) \sin( n\theta ) = \sin( (n-1)\theta ) + \sin( (n+1)\theta ) ,
$$
it follows that
\begin{equation}
\label{eqreqPn}
x P_{n-1} (x) = P_{n-2} (x) + P_n (x)
\hskip.5cm \text{for all} \hskip.2cm
n \ge 1.
\end{equation}
This implies, by induction on $n$, that $P_n$ is a polynomial,
of the form $P_n(x) = x^n + (\text{lower order terms})$ for all $n \ge 0$.
\par

The $P_n$~'s are Chebychev polynomials, up to a scale change.
More precisely, if $U_n(x)$ denotes the Chebychev polynomial of the second kind of degree~$n$,
defined by $U_n (\cos \theta) = \sin ((n+1)\theta) / \sin ( \theta )$,
then $P_n (x) = U_n( x/2)$.
\par

Define a probability measure $\mu$ on $\mathopen[ -2, 2 \mathclose]$ by
$$
d \mu (x) = \frac{1}{2\pi} \sqrt{ 4 - x^2} \hskip.1cm dx
\hskip.5cm \text{for} \hskip.2cm
x \in \mathopen[ -2, 2 \mathclose].
$$
Let $m, n \ge 0$;
using the change of variables $x = 2 \cos (\theta)$, we compute
$$
\begin{aligned}
&
\int_{-2}^2 P_m (x) P_n (x) d\mu(x)
= \frac{1}{2\pi} \int_{-2}^2 P_m (x) \sqrt{ 4 - x^2}
\hskip.1cm P_n (x) \sqrt{ 4 - x^2} \hskip.1cm \frac{dx}{ \sqrt{ 4 - x^2} }
\\
& \hskip1cm =
\frac{1}{2\pi} \int_0^\pi P_m(2 \cos (\theta)) 2 \sin (\theta)
\hskip.1cm P_n(2 \cos (\theta)) 2 \sin (\theta) \hskip.1cm d\theta
\\
& \hskip1cm =
\frac{2}{\pi} \int_0^\pi \sin ((m+1)\theta) \hskip.1cm \sin ((n+1)\theta) \hskip.1cm d\theta
\\
& \hskip1cm =
\frac{1}{\pi} \int_0^\pi \Big[
\cos \big[ (m+1)\theta - (n+1)\theta \big] - \cos \big[ (m+1)\theta) + (n+1)\theta \big]
\Big] d\theta
\\
& \hskip1cm =
\frac{1}{\pi} \int_0^\pi \Big[
\cos \big[ (m-n)\theta \big] - \cos \big[ (m+n+2)\theta \big]
\Big] d\theta
\\
& \hskip1cm =
0 \hskip.2cm \text{if} \hskip.2cm m \ne n
\hskip.2cm \text{and} \hskip.2cm
1 \hskip.2cm \text{if} \hskip.2cm m = n .
\end{aligned}
$$
It follows that $(P_n)_{n \ge 0}$
is an orthonormal basis of $L^2( \mathopen[ -2, 2 \mathclose], \mu )$.
If $M_\mu$ denotes the operator of multiplication by $x$ on this space,
we have by Equation~\eqref{eqreqPn} above
\begin{equation}
M_\mu P_n = P_{n-1} + P_{n+1}
\hskip.5cm \text{for all} \hskip.2cm
n \ge 0 ,
\end{equation}
where $P_{-1}$ should be read as $0$.
\par

This shows that $J$ is the matrix of $M_\mu$
with respect to the basis $(P_n)_{n \ge 0}$.
The claims of Proposition~\ref{infray} follow therefore
from the corresponding facts of Proposition~\ref{PropMO}.
\end{proof}

The \textbf{line} is the graph $L$ with vertex set $\Z = \{ \hdots , -1, 0, 1, \hdots\}$
and edge set $E = \{ \{j, j+1\} : j \in \Z \}$.
The line can be seen as the Cayley graph of the infinite cyclic group $\Z$
generated by $\{ 1, -1\}$.
The adjacency operator $A_L$ of $L$ is defined by
$$
(A_L \xi)(u)
= \xi(u-1) + \xi(u+1)
\hskip.5cm \text{for all} \hskip.2cm
\xi \in \ell^2(\Z) 
\hskip.2cm \text{and} \hskip.2cm
u \in \Z .
$$
The following proposition can be viewed as an exercise in Fourier series.

\begin{prop}
\label{infline}
Let $L$ be the infinite line and let $A_L$ be its adjacency operator.
\begin{enumerate}[label=(\arabic*)]
\item\label{1DEinfline}
The norm of $A_L$ is $2$.
\item\label{2DEinfline}
The spectrum of $A_L$ is
$\mathopen[ -2, 2 \mathclose]$.
\item\label{3DEinfline}
For all $j \in \Z$, 
the vertex spectral measure $\mu_j$ of $A_L$ at $j$ is given by its density
with respect to the Lebesgue measure:
$$
d \mu_j (x) = \frac{1}{\pi \sqrt{ 4 - x^2} } dx
\hskip.5cm \text{for} \hskip.2cm
x \in \mathopen[ -2, 2 \mathclose] .
$$
The measure $\mu_j$ is independent of $j$,
it is a scalar-valued spectral measure for $A_L$,
and the vertex $j$ is dominant.
\item\label{4DEinfline}
The operator $A_L$ is of uniform multiplicity $2$.
\end{enumerate}
\end{prop}

\begin{lem}
\label{Leminfline}
The adjacency operator $A_L$ of the line $L$
is unitarily equivalent to the operator
$M_{\mathopen[0, 2\pi \mathclose], \lambda, 2 \cos}$
of multiplication by the function $2 \cos$
on the Hilbert space $L^2 (\mathopen[0, 2\pi \mathclose], \lambda)$,
where $\lambda$ stands for the Lebesgue measure on the interval.
\end{lem}

\begin{proof}
The Fourier transform
$$
U \, \colon \ell^2 (\Z) \to L^2(\mathopen[0, 2\pi \mathclose], \lambda) ,
\hskip.5cm
(U \xi) (x) = \sum_{n \in \Z} \xi(n) e^{i n x} 
$$
is a surjective isometry with inverse
$$
U^{-1} \, \colon L^2(\mathopen[0, 2\pi \mathclose], \lambda) \to \ell^2 (\Z) ,
\hskip.5cm
(U^{-1} \eta) (n) = \frac{1}{ 2 \pi } \int_0^{2 \pi} \eta (x) e^{- i n x} dx .
$$
For any $\eta \in L^2(\mathopen[0, 2\pi \mathclose], \lambda)$, we have
$$
\begin{aligned}
\left( U A_L U^{-1} \eta \right) (x)
& =
\sum_{n \in \Z} \left( A_L U^{-1} \eta \right) (n) e^{i n x}
\\
& =
\sum_{n \in \Z} 
\Big(
\left( U^{-1} \eta \right) (n-1) e^{i n x} + \left( U^{-1} \eta \right) (n+1) e^{i n x}
\Big)
\\
& =
\Big( \sum_{k \in \Z} \left( U^{-1} \eta \right) (k) e^{i k x} \Big) e^{i x}
\hskip.1cm + \hskip.1cm
\Big( \sum_{k \in \Z} \left( U^{-1} \eta \right) (k) e^{i k x} \Big) e^{-i x}
\\
& =
\big( U (U^{-1}\eta) \big) (x) e^{i x} + \big( U (U^{-1}\eta) \big) (x) e^{-ix} =
2 \cos (x) \eta(x) 
\end{aligned}
$$
for all $x \in \mathopen[0, 2\pi \mathclose]$,
so that
$$
U A_L U^{-1} = M_{\mathopen[0, 2\pi \mathclose], \lambda, 2 \cos}
$$
as was to be proved.
\end{proof}

\begin{proof}[Proof of Proposition~\ref{infline}]
For~\ref{1DEinfline} and~\ref{2DEinfline},
use Lemma~\ref{Leminfline}:
the norm of $A_L$ is the norm of
$M_{\mathopen[0, 2\pi \mathclose], \lambda, 2 \cos}$,
which is $\sup_{-2 \le x \le 2} \vert 2 \cos (x) \vert = 2$,
and the spectrum of $A_L$
is the spectrum of $M_{\mathopen[0, 2\pi \mathclose], \lambda, 2 \cos}$,
which is the range of the function $2 \cos$,
namely which is $\mathopen[ -2, 2 \mathclose]$.

\vskip.2cm

\ref{3DEinfline}
Let $j \in \Z$, viewed as a vertex of $L$.
The vertex spectral measure $\mu_j$ at $j$ is defined by
$$
\int_{ \mathopen[ -2, 2 \mathclose] } f(x) d\mu_j (x)
= \langle f(A_L) \delta_j \sca \delta_j \rangle
$$
for all continuous function $f$ on the spectrum of $A_L$.
For $n \in \N$, its $n^{{\rm th}}$ moment is
$$
\int_{ \mathopen[ -2, 2 \mathclose] } x^n d\mu_j (x)
= \langle (A_L)^n \delta_j \sca \delta_j \rangle .
$$
This number is also the number of paths of length $n$
from $j$ to $j$ in the graph $L$.
When $n$ is odd, this number is clearly $0$.
When $n = 2m$ is even, each such path
has $m$ left steps and $m$ right steps,
so that this number is the binomial coefficient $\binom{2m}{m}$.
\par

The moments of the measure $\frac{ 1 }{ \pi \sqrt{4-x^2} } dx$
on $\mathopen[ -2, 2 \mathclose]$ are also easy to compute.
Moments of odd order vanish, because
$\int_{-2}^2 \frac{ f(x) }{ \pi \sqrt{4-x^2} } dx = 0$
when $f$ is an odd function,
in particular when $f(x) = x^{2m+1}$ for some $m \in \N$.
For moments of even order~$2m$,
using again the change of variables $x = 2 \cos \theta$,
we have
$$
\begin{aligned}
\int_{-2}^2 \frac{ x^{2m} }{ \pi \sqrt{4 - x^2} } \hskip.1cm dx
&=
\frac{1}{ \pi } \int_0^\pi
\frac{ (2 \cos \theta)^{2m} }{ 2 \sin \theta } 2 \sin \theta \hskip.1cm d\theta 
=
\frac{1}{ \pi } \int_0^\pi
\left( e^{i\theta} + e^{-i\theta} \right)^{2m} \hskip.1cm d\theta
\\
&=
\frac{1}{ \pi } \sum_{k = 0}^{2m} \binom{2m}{k} \int_0^\pi 
e^{i 2 (m-k) \theta} \hskip.1cm d\theta
=
\binom{2m}{m} ,
\end{aligned}
$$
because all but one term ($k=m$) vanish in the sum over $k$.
\par

These computations show that the measures $\mu_j$
and $\frac{ dx }{ \pi \sqrt{4-x^2} }$ on $\mathopen[ -2, 2 \mathclose]$
have the same moments, hence they are equal.
In particular $\mu_j$ is independent of $j$.
It follows from Proposition~\ref{PropDominant}~\ref{6DEPropDominant}
that this measure is a scalar-valued spectral measure for $A_L$,
and that the vertex $j$ is dominant.

\vskip.2cm

On the one hand, Claim~\ref{4DEinfline} 
follows from Proposition~\ref{PropMOmult}.
On the other hand, we prefer to show it with a more elementary argument,
as follows.
\par

We view the operator
$M_{\mathopen[0, 2\pi \mathclose], \lambda, 2 \cos}$
of Lemma~\ref{Leminfline}
as the direct sum of two operators:
the operator $M_{\mathopen[0, \pi \mathclose], \lambda, 2 \cos}$
of multiplication by $2 \cos$ on $L^2( \mathopen[ 0, \pi \mathclose], \lambda)$
and the operator $M_{\mathopen[\pi, 2\pi \mathclose], \lambda, 2 \cos}$
of multiplication by $2 \cos$ on $L^2( \mathopen[ \pi, 2 \pi \mathclose], \lambda)$.
By Example~\ref{ExCosEqt},
each of these two operators is unitarily equivalent
to the operator $M_1$ of multiplication by $x$
on $L^2( \mathopen[ -2, 2 \mathclose], \lambda)$.
It follows that $M_{\mathopen[0, 2\pi \mathclose], \lambda, 2 \cos}$,
and therefore also the adjacency operator $A_L$ of the line,
are unitarily equivalent to the operator of multiplication by $x$
on the space $L^2(\mathopen[ -2, 2 \mathclose], \lambda, \C^2)$,
so that $A_L$ is of uniform multiplicity $2$.
\end{proof}

Let now $d$ be an integer, $d \ge 1$.
Let $\{e_1, \hdots, e_d\}$ be the canonical basis
of the free abelian group $\Z^d$.
The \textbf{lattice} $L_d$ is the graph
with vertex set $\Z^d$ and edge set
$$
E = \{ \{u, v\} : u \in \Z^d, \hskip.1cm v = u + e_j
\hskip.2cm \text{for some} \hskip.2cm
j \in \{1, \hdots, d \} \} .
$$
In other words, $L_d$ is the Cayley graph of the group~$\Z^d$
with respect to the generating set $\{ \pm e_1, \hdots, \pm e_d \}$.
The adjacency operator $A_d$ of $L_d$ is given by
$$
(A_d \xi)(u) = \sum_{j = 1}^d \xi(u-e_j) + \xi(u+e_j)
\hskip.5cm \text{for all} \hskip.2cm
\xi \in \ell^2(\Z^d) 
\hskip.2cm \text{and} \hskip.2cm
u \in \Z^d .
$$
When $d=1$, the lattice $L_1$ is the infinite line $L$
of Proposition~\ref{infline};
now we denote by $\mu_1$ the vertex spectral measure of the line,
given by
$d \mu_1 (x) = \frac{1}{\pi \sqrt{ 4 - x^2} } dx$
for all $x \in \mathopen[ -2, 2 \mathclose]$.

\begin{prop}
\label{lattices}
Let $d \ge 2$. Let $L_d$ be the lattice graph of dimension $d$
and let $A_d$ be its adjacency operator.
\begin{enumerate}[label=(\arabic*)]
\item\label{1DElattices}
The norm of $A_d$ is $2$.
\item\label{2DElattices}
The spectrum of $A_d$ is
$\mathopen[ -2d, 2d \mathclose]$.
\item\label{3DElattices}
The vertex spectral measure $\mu_d$ of a vertex $v$ in $L_d$
is independent of $v$; it is the convolution of $d$ copies
of the spectral measure $\mu_1$ of Proposition~\ref{infline}.
It is a scalar-valued spectral measure for $A_d$
and it is equivalent to
the Lebesgue measure supported on $\mathopen[ -2d, 2d \mathclose]$.
\item\label{4DElattices}
The operator $A_d$ had infinite uniform multiplicity.
\end{enumerate}
\end{prop}

For the proof of the proposition above,
we begin as for Proposition~\ref{infline},
with minor modifications.
Much of what follows holds for $d \ge 1$, rather than for $d \ge 2$ only.
Proposition~\ref{PropMOmult} is used for
the only slightly delicate point, which is our proof of~\ref{4DElattices}.

\begin{lem}
\label{Lemlattices}
Let $\lambda$ denote the Lebesgue measure
on $\mathopen[0, 2\pi \mathclose]^d$.
The Fourier transform
$$
U \, \colon \ell^2 (\Z^d) \to L^2(\mathopen[0, 2\pi \mathclose]^d, \lambda) ,
\hskip.5cm
(U \xi) (x) = \sum_{u \in \Z^d} \xi(u) e^{i \langle u \mid x \rangle} 
$$
(where $\langle u \sca x \rangle = \sum_{j=1}^d u_j x_j$)
is a surjective isometry with inverse
$$
U^{-1} \, \colon L^2(\mathopen[0, 2\pi \mathclose]^d, \lambda) \to \ell^2 (\Z^d) ,
\hskip.5cm
(U^{-1} \eta) (u) =
\frac{1}{ (2 \pi)^d } \int_{ \mathopen[ 0, 2 \pi \mathclose]^d }
\eta (x) e^{- i \langle u \mid x \rangle } dx .
$$
Let $2 \sum \cos$ be the function
$\mathopen[0, 2\pi \mathclose]^d \to \R , \hskip.2cm
x = ( x_j )_{1 \le j \le d} \mapsto 2 \sum_{j=1}^d \cos (x_j)$.
\par

The operators $A_d$ and $M_{\mathopen[0, 2\pi \mathclose]^d, \lambda, 2\sum \cos}$
are unitarily equivalent; more precisely:
$$
U A_d U^{-1} = M_{\mathopen[0, 2\pi \mathclose]^d, \lambda, 2\sum \cos} .
$$
\end{lem}

\begin{proof}
For any $\eta \in L^2(\mathopen[0, 2\pi \mathclose]^d, \lambda)$, we have
$$
\begin{aligned}
& \left( U A_d U^{-1} \eta \right) (x) =
\sum_{u \in \Z^d} \left( A_d U^{-1} \eta \right) (u) e^{i \langle u \mid x \rangle }
\\
& \hskip.5cm =
\sum_{u \in \Z^d} \sum_{j=1}^d
\Big( \left( U^{-1} \eta \right) (u-e_j) e^{i \langle u \mid x \rangle }
+ \left( U^{-1} \eta \right) (u+e_j) e^{i \langle u \mid x \rangle} \Big)
\\
& \hskip.5cm =
\sum_{j=1}^d
\Big( \sum_{k \in \Z^d} 
\left( U^{-1} \eta \right) (k) e^{i \langle k \mid x \rangle } \Big) e^{i x_j}
\hskip.1cm + \hskip.1cm \sum_{j=1}^d
\Big( \sum_{k \in \Z^d} 
\left( U^{-1} \eta \right) (k) e^{i \langle k \mid x \rangle } \Big) e^{-i x_j}
\\
& \hskip.5cm =
\sum_{j=1}^d U (U^{-1} \eta) (x) e^{i x_j}
+ \sum_{j=1}^d U (U^{-1} \eta) (x) e^{-i x_j}
= \left( 2 \sum_{j=1}^d \cos (x_j) \right) \eta(x) ,
\end{aligned}
$$
so that $U A_L U^{-1}$ is the operator of multiplication by $2 \sum \cos$
on the Hilbert space $L^2(\mathopen[0, 2\pi \mathclose]^d, \lambda)$.
\end{proof}

\begin{proof}[Proof of Proposition~\ref{lattices}]
By Proposition~\ref{PropMO},
the operator $M_{\mathopen[0, 2\pi \mathclose]^d, \lambda, 2\sum \cos}$
of multiplication by the function
$2 \sum_{j=1}^d \cos (x_j)$ on $L^2(\mathopen[0, 2\pi \mathclose]^d, \lambda)$
has norm $2d$ and spectrum $\mathopen[ -2d, 2d \mathclose]$.
Claims \ref{1DElattices} and \ref{2DElattices} follow from Lemma~\ref{Lemlattices}.
\par

Observe that there is a natural isomorphism
$$
\ell^2(\Z) \otimes \ell^2(\Z) \otimes \cdots \otimes \ell^2(\Z) \to \ell^2(\Z^d)
$$
by which we can identify the operators
$$
A_1 \otimes {\rm Id} \otimes \cdots \otimes {\rm Id} + \cdots
+ {\rm Id} \otimes \cdots \otimes {\rm Id} \otimes A_1
\hskip.5cm \text{and} \hskip.5cm
A_d .
$$
By Proposition~\ref{localmeasuretensorproduct},
the vertex spectral measure of $A_d$ at a vertex of $L_d$
is the convolution of $d$ copies
of the vertex spectral measure of $A_1$ at a vertex of $L_1$.
It follows from Proposition~\ref{PropDominant}~\ref{6DEPropDominant}
that the vertex spectral measure of $A_d$ at a vertex of $L_d$
is a scalar-valued spectral measure for $A_d$.
This proves the first part of Claim \ref{3DElattices}.
\par

By Proposition~\ref{infline},
the vertex spectral measure at a vertex of the line $L_1$ is
$d\mu_1(x) = f(x) dx$,
where $f(x) = \frac{1}{ \pi \sqrt{4 - x^2} }$ if $-2 < x < 2$
and $f(x) = 0$ otherwise,
and where $\lambda$ stands for the Lebesgue measure.
The vertex spectral measure at a vertex of the lattice $L_d$,
which is the convolution power
$\mu_d \Doteq \mu_1 \ast \mu_1 \ast \cdots \ast \mu_1$ ($d$~factors),
is consequently of the form $f_d (x) dx$,
where $f_d$ is a continuous function,
$f_d(x) > 0$ for all $x \in \mathopen] -2d, 2d \mathclose[$,
and $f_d(x) = 0$ for all $x$ such that $\vert x \vert \ge 2d$.
In particular, this measure $\mu_d$ is equivalent to the Lebesgue measure
on the interval $\mathopen[ -2d, 2d \mathclose]$.
This concludes the proof of Claim \ref{3DElattices}.
\par

By Proposition~\ref{PropMOmult},
the operator $M_{\mathopen[0, 2\pi \mathclose]^d, \lambda, 2\sum \cos}$
has uniform infinite spectral multiplicity.
By Lemma~\ref{Lemlattices},
Claim \ref{4DElattices} follows.
\end{proof}

\begin{rem}
\label{remLonRd}
Consider the so-called \emph{discrete Laplacian}
$D_d = 2d \hskip.1cm {\rm Id} - A_d$
on the lattice $L_d$, acting on $\ell^2(\Z^d)$.
Proposition~\ref{lattices} shows that
$D_d$ has spectrum $\mathopen[ 0, 4d \mathclose]$
and uniform multiplicity, $2$ when $d=1$ and $\infty$ when $d \ge 2$.
The \emph{continuous} Laplacian $\Delta_d = - \sum_{j=1}^d \frac{ \partial^2 }{ \partial x_j^2 }$
on the Euclidean space $\R^d$
is an unbounded self-adjoint operator with domain
$$
\emph{Dom} (\Delta_d) =
\left\{ \xi \in L^2(\R^d, \lambda)
\hskip.2cm : \hskip.2cm
\int_{\R^d} \Vert k \Vert^2 \hskip.1cm \vert \widehat \xi (k) \vert^2
d \lambda (k) < \infty \right\} ,
$$
where $\lambda$ denotes the Lebesgue measure
and $\widehat \xi$ the Fourier transform of $\xi$.
The spectrum of $\Delta_d$ is $\mathopen[ 0 , \infty \mathclose[$.
The operators $D_d$ and $\Delta_d$ share the same multiplicities:
it is known that $\Delta_d$ has uniform multiplicity,
$2$ when $d=1$ and $\infty$ when $d \ge 2$.
\end{rem}

Let $D_\infty = (V', E')$ be the graph obtained from $R = (V, E)$
by adding one vertex $0'$ to the set of vertices $V$ of $R$
and one edge $\{0', 1\}$ to the set of edges $E$ of $R$.
Thus $V' = \{0'\} \cup \N$ and 
$E' = \big\{ \{0', 1\}, \{0,1\}, \{1,2\}, \{2, 3\}, \hdots \big\} = \{0', 1\} \cup E$.
Let $A_D$ be the adjacency operator of $D_\infty$.

\begin{prop}
\label{PropSpecD}
The spectrum of $A_D$ is $\mathopen[ -2, 2 \mathclose]$
and $0$ is an eigenvalue of $A_D$.
The vertices $0$ and $0'$ are cyclic,
and $A_D$ is multiplicity-free.
\end{prop}

\begin{proof}
The spectrum $\Sigma (X)$ of a bounded self-adjoint operator $X$
is the union of the essential spectrum $\Sigma_{{\rm ess}} (X)$
and a discrete set of points in $\R \smallsetminus \Sigma_{{\rm ess}} (X)$
which are eigenvalues of finite multiplicity.
In particular
$\Sigma_{{\rm ess}} (J_1) = \Sigma(J_1) = \mathopen[ -2, 2 \mathclose]$
by Proposition~\ref{infray}.
\par

Let $R' = (V', E)$ be the graph obtained from $R = (V, E)$
by adding one isolated vertex $\{0'\}$,
and let $A'_R$ be its adjacency operator.
The marked spectrum of $A'_R$ is the union of that of $A_R = J_1$
and of the simple eigenvalue $0$.
The operator $A_D$ is a perturbation of $A'_R$
by an operator of finite rank, indeed of rank $2$.
If $K$ is a compact self-adjoint operator on the same space as $X$,
it is a theorem of Weyl that
$\Sigma_{{\rm ess}} (X+K) = \Sigma_{{\rm ess}} (X)$
\cite[Theorem 3.14.1]{Sim4--15}.
In particular
\begin{equation}
\label{eqSigmaEssAD}
\Sigma_{{\rm ess}}(A_D)
\overset{\text{(by Weyl)}}{=} \Sigma_{{\rm ess}}(A'_R)
= \Sigma(A'_R) = \mathopen[ -2, 2 \mathclose] .
\end{equation}
\par

Let $n \ge 4$.
The finite graph $D_n$ has vertex set $\{0', 0, 1, \hdots, n-2\}$
and edge set $\Big\{ \{0',1\}, \{0,1\}, \{1,2\}, \hdots, \{n-3,n-2\} \Big\}$.
The spectrum $\Sigma(D_n)$ of its adjacency operator
is well-known
\cite[Theorem 3.1.3]{BrHa--12}
to be a finite subset of $\mathopen] -2, 2 \mathclose[$.
Let $D'_n$ be the graph with vertex set $V'$ and the same edge set as $D_n$.
Since $0 \in \Sigma(D_n)$, the spectrum of $D'_n$ is the same as that of $D_n$.
For $n \to \infty$, the sequence of the adjacency operators of $D'_n$
converges strongly to $A_D$.
It follows that $\Sigma(A_D)$ is contained in the union
$\bigcup_{n \ge 4} \Sigma(D'_n)$, hence in $\mathopen[ -2, 2 \mathclose]$;
see~\cite[Section~X.7]{DuSc--63}.
Together with~\eqref{eqSigmaEssAD},
this shows that $\Sigma(A_D) = \mathopen[ -2, 2 \mathclose]$.
\par

Let $\xi \in \ell^2(V')$ be defined by
$\xi(0) = 1$, $\xi(0') = -1$ and $\xi(j) = 0$ for all $j \ge 1$.
Then $A_D \xi = 0$, so that $0$ is an eigenvalue of $A_D$.
It is easy to check that $0$ and $0'$ are cyclic vertices;
if necessary, see \cite[Example 7.2]{BrHN}.
\end{proof}

\begin{prop}
\label{propnormAis2}
Let $G$ be an infinite connected graph of bounded degree
with adjacency operator $A_G$ such that $\Vert A_G \Vert \le 2$.
Then $\Vert A_G \Vert = 2$, $\Sigma (A_G) = \mathopen[ -2, 2 \mathclose]$,
and $G$ is isomorphic to one of the three following graphs:
\begin{enumerate}
\item[--]
the infinite ray $R$ and then $A_G$ is multiplicity-free,
without eigenvalue,
\item[--]
the graph $D_\infty$ and then $A_G$ is multiplicity-free, 
with an eigenvalue,
\item[--]
the infinite line $L$ and then $A_G$ is of uniform multiplicity two,
without eigenvalue.
\end{enumerate}
It follows that these three graphs are determined by their marked spectrum
among connected graphs of bounded degree.
\end{prop}

\begin{proof}
Let $F = (V_F, E_F)$ be a finite subgraph of $G = (V_G, E_G)$,
and let $F_{{\rm ind}} = (V_F, E_{{\rm ind}})$
be the subgraph of $G$ induced by $V_F$.
Then $\Vert A_F \Vert \le \Vert A_{F_{{\rm ind}}} \Vert$
by Perron--Frobenius theory and
$\Vert A_{F_{{\rm ind}}} \Vert \le \Vert A_G \Vert$
by standard arguments
(details in \cite[proof of Proposition 3.1]{BrHN}),
so that $\Vert A_F \Vert \le 2$.
\par

Computations with finite graphs show
we would have $\Vert A_F \Vert > 2$
if $F$ was a connected finite graph containing strictly one of
$\widetilde A_n$ ($n \ge 2$),
$\widetilde D_n$ ($n \ge 4$),
$\widetilde E_n$ ($n = 6,7,8$),
and this is not possible.
Here $\widetilde A_n$ denotes the cycle with $n+1$ vertices,
$\widetilde D_n$ the graph obtained
from a segment with vertices $v_1, \hdots, v_{n-1}$
and edges $\{v_j, v_{j+1}\}$ ($1 \le j \le n-2$)
by adding two vertices $v_0, v_n$
and two edges $\{v_0, v_2\}$, $\{v_{n-2}, v_n \}$,
and $\widetilde E_6, \widetilde E_7, \widetilde E_8$
the stars with respectively $7, 8, 9$ vertices
described in \cite{BrHa--12};
see Theorem 3.1.3 in this book.
It follows that $G$ is a tree,
because it does not contain strictly any $\widetilde A_n$
($n \ge 2$).
Also $G$ does not have vertices of degree $\ge 4$,
and $G$ has at most one vertex of degree $3$,
because it does not contain strictly any $\widetilde D_n$
($n \ge 4$).
And finally, if $G$ contains a vertex of degree $3$,
two of the segments starting from this vertex must be of length $1$, 
because it does not contain strictly any $\widetilde E_n$
($n=6,7,8$).
It follows that $G$ is isomorphic to one of $R, D_\infty, L$,
hence $\Vert A_G \Vert = 2$
and $\Sigma (A_G) = \mathopen[ -2, 2 \mathclose]$.
\par

Since $R, D_\infty$, and $L$ have different multiplicity functions,
each of them is determined by its marked spectrum among connected graphs.
\end{proof}

\begin{prop}
\label{propRischarac}
The infinite ray $R$ is determined by its marked spectrum.
\end{prop}

\begin{proof}
Let $G$ be a graph of bounded degree with the same marked spectrum
as that of $R$.
Let $(G_i)_{i \in I}$ be the connected components of $G$.
Denote by $A_G$ the adjacency operator of $G$
and, for each $i \in I$, by $A_i$ that of $G_i$.
There cannot exist $i \in I$ with $G_i$ finite;
otherwise $\Sigma(A_i)$ would consist of eigenvalues,
and thus $\Sigma(A_G)$ would contain eigenvalues,
but this is impossible since $\Sigma(A_R)$ does not;
hence each $G_i$ is infinite.
By Proposition~\ref{propnormAis2},
each $G_i$ is isomorphic to one of $R, D_\infty$, or $L$;
but $D_\infty$ is impossible because $A_R$ does not have eigenvalue
and $L$ is impossible because $A_R$ has uniform spectral multiplicity $1$;
hence each $G_i$ is isomorphic to $R$.
The graph $G$ cannot be the union of $2$ or more
connected components isomorphic to $R$,
again because $A_R$ has uniform spectral multiplicity~$1$;
hence $G$ is isomorphic to $R$.
\end{proof}

Note that the infinite line $L$ is not characterized by its marked spectrum.
Indeed, the adjacency operator $A_L$ and the adjacency operator
of a graph with two connected components isomorphic to $R$
are unitarily equivalent,
as it follows from Corollary~\ref{CorRefUnitEq}
and Propositions~\ref{infray} \&~\ref{infline}.

It is natural to ask whether there are other infinite connected graphs $G$
with $\Vert A_G \Vert < \sqrt{ 2 + \sqrt 5 } \sim 2.058$
which are characterized by their marked spectrum
among connected graphs; see \cite{BrNe--89}.
The range $\sqrt{ 2 + \sqrt 5 } \le \Vert A_G \Vert \le \frac{3}{2} \sqrt 2 \sim 2.121$
could also be investigated \cite{WoNe--07}.

\section{\textbf{Spherically symmetric infinite rooted trees}}
\label{SectionSSRTree}

Let $T = (V,E)$ be a spherically symmetric rooted tree,
of bounded degree and without leaves,
and let $A_T$ be its adjacency operator.
The main technical result of this section is Proposition~\ref{PropAeqJ},
showing an orthogonal decomposition of $\ell^2(V)$
in subspaces invariant by $A_T$
on each of which $A_T$ is an infinite Jacobi matrix.
This is standard, it has been used for trees as here and in other contexts;
see \cite{RoRu--95},
\cite[Lemma 1]{AlFr--00},
\cite[Theorem 3.2]{Solo--04},
\cite[Theorem 2.4]{Breu--07}. 
%
\par

Let $T = (V, E)$ be a tree.
Choose a root $v_0 \in V$.
For $v \in V$, denote by $\vert v \vert$ the distance from $v$ to $v_0$.
For an integer $r \ge 0$, let $S_r = \{ v \in V : \vert v \vert = r \}$
be the sphere in $V$ of radius $r$ around $v_0$.
For $v \in V$, denote by $N^+_v$ the set of neighboring vertices of $V$
at distance $\vert v \vert + 1$ from $v_0$.
For $v \in V$ different from $v_0$,
denote by $v_-$ the neighboring vertex of $v$
at distance $\vert v \vert - 1$ from $v_0$;
note that, for $v \ne v_0$, the set of neighbors of $v$ is $\{ v_- \} \cup N^+_v$,
and therefore the degree of $v$ is $\deg (v) = 1 + \vert N^+_v \vert$.
The set of neighbors of $v_0$ is $N^+_{v_0} = S_1$.
\par

The infinite rooted tree $T$ is \textbf{spherically symmetric} if,
for every $r \ge 0$, every vertex in $S_r$
has exactly $d_r \ge 1$ adjacent vertices in $S_{r+1}$,
for some sequence $(d_r)_{r \ge 0}$ of positive integers,
the sequence of \textbf{branching degrees} of $T$.
From now on,
we consider an infinite spherically symmetric rooted tree $T$
of bounded degree,
with sequence of branching degrees such that
\begin{equation}
\label{eqsequencedegreesT}
d_r \ge 2
\hskip.2cm \text{for all} \hskip.2cm
r \ge 0
\hskip.5cm \text{and} \hskip.5cm
\sup_r d_r < \infty .
\end{equation}
For $r \ge 0$, we identify $\ell^2(S_r)$ with the subspace of $\ell^2 (V)$
of functions which vanish on $V \smallsetminus S_r$.
We set $\ell^2 (S_{-1}) = \{0\}$.
Define an operator $H$ on $\ell^2(V)$ by
\begin{equation}
\label{eqdefofH}
(H \xi) (v) = \xi(v_-)
\hskip.2cm \text{if} \hskip.2cm
\vert v \vert \ge 1 
\hskip.5cm \text{and} \hskip.5cm
(H \xi) (v_0) = 0
\hskip.5cm \text{for all} \hskip.2cm
\xi \in \ell^2(V) .
\end{equation}

\begin{prop}
\label{PropHandH*}
Let $T = (V, E)$ be a spherically symmetric infinite rooted tree
with root $v_0 \in V$, 
and with sequence of branching degrees $(d_r)_{r \ge 0}$
such that Condition~\eqref{eqsequencedegreesT} holds.
Let $A_T$ and $H$ be as above.
\begin{enumerate}[label=(\arabic*)]
\item\label{1DEPropHandH*}
The operator $H$ is bounded on $\ell^2(V)$
of norm $\sqrt{ \max_{r \ge 0} d_r }$,
and is injective.
\item\label{2DEPropHandH*}
The adjoint $H^*$ of $H$ is given by
\begin{equation}
\label{eqdefH*}
(H^* \xi) (v) = \sum_{w \in N^+_v} \xi (w)
\hskip.5cm \text{for all} \hskip.2cm
\xi \in \ell^2(V)
\hskip.2cm \text{and} \hskip.2cm
v \in V ,
\end{equation}
and we have
$$
A_T = H + H^* .
$$
\item\label{3DEPropHandH*}
For all $r \ge 0$:
\begin{enumerate}
\item[$\circ$]
the restriction $\frac{1}{ \sqrt{d_r}} H \big\vert_{\ell^2(S_r)}$
is an isometry from $\ell^2(S_r)$ into $\ell^2(S_{r+1})$
\item[]
and $\frac{1}{ d_r } H^*H \big\vert_{\ell^2(S_r)} = {\rm Id}_{\ell^2(S_r)}$,
\item[$\circ$]
$H^*(\ell^2(S_r)) = \ell^2(S_{r-1})$
and $HH^*(\ell^2(S_r)) \subset \ell^2(S_r)$.
\end{enumerate}
\item\label{4DEPropHandH*}
Let $r \ge 0$ and $k \ge 0$.
If $\xi$ and $\eta$ in $\ell^2( S_r )$ are orthogonal,
then $H^k \xi$ and $H^k \eta$ in $\ell^2( S_{r+k} )$ are also orthogonal.
\end{enumerate}
\end{prop}

\begin{proof}
\ref{1DEPropHandH*}
Let $\xi \in \ell^2(V)$. We have
$$
\begin{aligned}
\Vert H \xi \Vert^2
& = \sum_{v \in V} \vert (H \xi) (v) \vert^2
= \sum_{v \in V, v \ne v_0} \vert \xi (v_-) \vert^2
= \sum_{w \in V} d_{\vert w \vert} \hskip.1cm \vert \xi (w) \vert^2
\\
& \le \big( \max_{r \ge 0} d_r \big) \sum_{w \in V} \vert \xi (w) \vert^2
= \big( \max_{r \ge 0} d_r \big) \Vert \xi \Vert^2 ,
\end{aligned}
$$ 
hence $\Vert H \Vert \le \sqrt{ \max_{r \ge 0} d_r }$.
For the equality, see the end of~\ref{3DEPropHandH*} below.
\par

If $H\xi = 0$, i.e., if $\xi(v_-) = 0$
for all $v \in V \smallsetminus \{ v_0 \}$, then $\xi = 0$,
hence $H$ is injective.

\vskip.2cm

\ref{2DEPropHandH*}
We use temporarily Formula~\eqref{eqdefH*}
as a definition of $H^*$.
Then $H^*$ is bounded;
indeed, using the Cauchy--Schwarz inequality,
we have for all $\xi \in \ell^2(V)$
$$
\begin{aligned}
\sum_{v \in V} \vert (H^* \xi) (v) \vert^2
& = \sum_{v \in V} \Big\vert \sum_{w \in N^+_v} \xi (w) \Big\vert^2
\le \sum_{v \in V} d_{\vert v \vert} \sum_{w \in N^+_v} \vert \xi (w) \vert^2
\\
& = \sum_{w \in V, w \ne v_0} d_{\vert w \vert - 1} \vert \xi (w) \vert^2
\le \big( \max_{r \ge 0} d_r \big) \Vert \xi \Vert^2 .
\end{aligned}
$$
And $H^*$ is the adjoint of $H$ because,
for $\xi, \eta \in \ell^2(V)$, we have
$$
\begin{aligned}
\langle H^* \xi \sca \eta \rangle
& = \sum_{v \in V} (H^*\xi)(v) \overline{\eta(v)}
= \sum_{v \in V} \sum_{w \in N^+_v} \xi(w) \overline{ \eta (v) }
\\
& = \sum_{w \ne v_0} \xi(w) \overline{ \eta (w_-) }
= \sum_{w \ne v_0} \xi(w) \overline{ (H \eta) (w) }
= \langle \xi \sca H \eta \rangle .
\end{aligned} 
$$
The equality $A_T = H + H^*$ follows
from~\eqref{eqdefofH} and~\eqref{eqdefH*}.

\vskip.2cm

\ref{3DEPropHandH*}
Let $\xi \in \ell^2(S_r)$. It is obvious that $H\xi \in \ell^2(S_{r+1})$.
Moreover, the computation of the proof of~\ref{1DEPropHandH*}
continues as
$$
\Vert H \xi \Vert^2
= \sum_{w \in V} d_{\vert w \vert} \hskip.1cm \vert \xi (w) \vert^2
= d_r \sum_{w \in S_r} \vert \xi (w) \vert^2
= d_r \Vert \xi \Vert^2 ,
$$
hence $\frac{1}{ \sqrt{d_r}} H \big\vert_{\ell^2(S_r)}$
is an isometry from $\ell^2(S_r)$ into $\ell^2(S_{r+1})$.
We have also
$$
(H^*H\xi) (v) = \sum_{w \in N^+_v} (H\xi) (w) = d_r \xi (v)
\hskip.5cm \text{for all} \hskip.2cm
v \in V ,
$$
hence
\begin{equation}
\label{eqH*H}
\frac{1}{ d_r } H^*H \big\vert_{\ell^2(S_r)} = {\rm Id}_{\ell^2(S_r)} .
\end{equation}
It follows that $H^*$ maps $\ell^2(S_{r+1})$ onto $\ell^2(S_r)$,
and also that $\Vert H \Vert \ge \sqrt{d_r}$.
\par
It follows now that $\Vert H \Vert = \sqrt{ \max_{r \ge 0} d_r }$.

\vskip.2cm

\ref{4DEPropHandH*}
For $\xi$ and $\eta$ orthogonal in $\ell^2(S_r)$
we have, using Equality~\eqref{eqH*H},
$$
\langle H\xi \sca H\eta \rangle
= \langle H^*H \xi \sca \eta \rangle
= d_r \langle \xi \sca \eta \rangle = 0 ,
$$
so that $H\xi$ and $H\eta$ are orthogonal in $\ell^2(S_{r+1})$.
For $k \ge 2$, the same argument repeated $k$ times
shows that $H^k \xi$ and $H^k \eta$ are orthogonal.
\end{proof}

Set
$$
\Ui_{0,0} = \ell^2(S_0)
\hskip.5cm \text{and} \hskip.5cm
\Ui_{0, r} = H^r (\Ui_{0, 0})
\hskip.5cm \text{for each integer} \hskip.2cm
r \ge 0.
$$
Note that $\Ui_{0, r}$ is the one-dimensional subspace of $\ell^2(V)$
of functions on $V$ which vanish outside $S_r$
and which are constant on $S_r$.
Set
$$
\Vi_0 = \bigoplus_{r = 0}^\infty \Ui_{0, r} ,
$$
which is the subspace of $\ell^2(V)$
of functions which are constant on each sphere. 
\par

We define now subspaces $\Ui_{n, r}$ and $\Vi_n$
for $n \ge 1$ and $r \ge n$, by induction on~$n$.
Let $n \ge 1$; assume that $\Ui_{m, q}$ has already been defined
when $0 \le m < n$ and $q \ge m$.
Define
$$
\begin{aligned}
\Ui_{n, n} &= \text{
orthogonal complement
of $\Ui_{0, n} \oplus \Ui_{1, n} \oplus \cdots \oplus \Ui_{n-1, n}$
in $\ell^2(S_n)$,
}
\\
\Ui_{n,r} &= H^{r-n} (\Ui_{n, n})
\hskip.2cm \text{in} \hskip.2cm
\ell^2(S_r)
\hskip.2cm \text{for all} \hskip.2cm
r \ge n ,
\\
\Vi_n &= \bigoplus_{r = n}^\infty \Ui_{n, r} .
\end{aligned}
$$
Observe that
\begin{equation}
\label{eqVSU}
\ell^2(V) = \bigoplus_{r=0}^\infty \ell^2(S_r)
\hskip.5cm \text{and} \hskip.5cm
\ell^2(S_r) = \bigoplus_{n=0}^r \Ui_{n,r}
\hskip.5cm \text{for all} \hskip.2cm
r \ge 0 .
\end{equation}

\begin{prop}
\label{PropDirSum}
Let the notation be as above.
There are orthogonal direct sums decompositions
$$
\ell^2(V) = \bigoplus_{n=0}^\infty \Vi_n
= \bigoplus_{n=0}^\infty \bigoplus_{r=n}^\infty \Ui_{n, r} .
$$
For each $n \ge 0$, the subspace $\Vi_n$ of $\ell^2(V)$
is invariant by $H$, $H^*$, and $A_T$.
\end{prop}

\begin{proof}
We continue to follow \cite{AlFr--00}.
\par

We first check that the direct sums are orthogonal.
Let $n_1, r, s$ be nonnegative integers such that
$r \ne s$ and $0 \le n_1 \le \min\{r,s\}$.
The spaces $\Ui_{n_1,r}$ and $\Ui_{n_1,s}$ are orthogonal,
because they are respectively subspaces of $\ell^2(S_r)$ and $\ell^2(S_s)$
which are orthogonal.
It follows that $\Vi_{n_1} = \bigoplus_{r = n_1}^\infty \Ui_{n_1, r}$
is an orthogonal sum.
Let moreover $n_2$ be an integer such that $n_2 > n_1$.
The spaces $\Ui_{n_1,n_2}$ and $\Ui_{n_2,n_2}$
are orthogonal by definition of $\Ui_{n_2,n_2}$. 
By~\ref{4DEPropHandH*} of Proposition~\ref{PropHandH*},
the spaces $\Ui_{n_1, r} = H^{r-n_2}(\Ui_{n_1, n_2})$
and $\Ui_{n_2, r} = H^{r-n_2}(\Ui_{n_2, n_2})$
are orthogonal whenever $r \ge n_2$.
It follows that $\Vi_{n_1} = \bigoplus_{r = n_1}^\infty \Ui_{n_1, r}$
and $\Vi_{n_2} = \bigoplus_{r = n_2}^\infty \Ui_{n_2, r}$
are orthogonal,
and therefore that $\ell^2(V) = \bigoplus_{n=0}^\infty \Vi_r$
is an orthogonal sum.
\par

By definition, each $\Vi_n$ is invariant by $H$.
It remains to show that each $\Vi_n$ is also invariant by $H^*$,
i.e., that $H^*(\Ui_{n,r}) \subset \Vi_n$ for all $r \ge n$.

\vskip.2cm

Let $\xi \in \Ui_{n, r}$ for some $n$ and $r$ such that $0 \le n \le r$;
we distinguish three cases.
\par

Assume first that $r > n$.
There exists $\eta \in \Ui_{n, n}$ such that $\xi = H^{r-n} \eta$.
Then $H^* \xi = (H^*H) (H^{r-n-1} \eta) = d_{r-1} H^{r-n-1} \eta$
by~\ref{3DEPropHandH*} of Proposition~\ref{PropHandH*},
hence $H^* \xi \in \Ui_{n, r-1} \subset \Vi_n$.
\par

Assume now that $r = n \ge 1$.
Then $H^* \xi \in \ell^2 (S_{n-1})$.
We claim that $H^* \xi = 0$.
Indeed, choose $\ell \in \{0, 1, \hdots, n-1\}$
and $\zeta \in \Ui_{\ell, n-1}$.
Then $H \zeta \in \Ui_{\ell, n}$ and $\xi \in \Ui_{n, n}$ are orthogonal
(because $\ell < n$),
so that $\langle H^* \xi \sca \zeta \rangle = \langle \xi \sca H \zeta \rangle = 0$;
hence $H^* \xi$ is orthogonal to $\Ui_{\ell, n-1}$ for each $\ell \le n-1$,
i.e., $H^* \xi$ is orthogonal to $\ell^2(S_{n-1})$, i.e., $H^* \xi = 0$.
\par

Assume finally that $r = n = 0$; then $H^* \xi = 0$.
This shows that $H^* \xi \in \Vi_n$ in all cases.
\end{proof}

The next proposition is now straightforward:

\begin{prop}
\label{PropHandH*suite}
With the notation as above, we have
\begin{enumerate}[label=(\arabic*)]
\item\label{5DEPropHandH*suite}
$\dim \ell^2(S_n) = \vert S_n \vert = \prod_{q=0}^{n-1} d_q$ for all $n \ge 0$,
\item\label{6DEPropHandH*suite}
$\dim \Ui_{n,r} = \Big( \prod_{q=0}^{n-2} d_q \Big) (d_{n-1} - 1)$
for all $n \ge 2$ and $r \ge n$,
\newline
and $\dim \Ui_{1,r} = d_0 - 1$ for all $r \ge 1$,
and $\dim \Ui_{0,r} = 1$ for all $r \ge 0$,
\item\label{7DEPropHandH*suite}
$\dim \Vi_n = \infty$ for all $n \ge 0$.
\end{enumerate}
\end{prop}

Let $n \ge 0$.
Denote by $\ell^2(\N, \Ui_{n, n})$ the Hilbert space of
sequences $(\xi_j)_{j \ge 0}$ of vectors in $\Ui_{n, n}$
such that $\sum_{j=0}^\infty \Vert \xi_j \Vert^2 < \infty$.
For all $j \ge 0$,
by~\ref{3DEPropHandH*} of Proposition~\ref{PropHandH*}
and by definition of $\Ui_{n, n+j}$,
the operator
$$
\frac{1}{ \sqrt{ \prod_{q=n}^{n+j-1} d_q } } \hskip.1cm H^j \, \colon
\Ui_{n,n} \to \Ui_{n, n+j}
$$
is a surjective isometry.
\par

Let $\xi \in \Vi_n$.
For all $j \ge 0$,
there exists $\xi_{n+j} \in \Ui_{n, n+j}$,
and therefore $\chi_{n+j} \in \Ui_{n,n}$,
such that
\begin{equation}
\label{Seq}
\xi = \big( \xi_{n+j} \big)_{j \ge 0}
\hskip.5cm \text{with} \hskip.5cm
\xi_{n+j} =
\frac{1}{ \sqrt{ \prod_{q=n}^{n+j-1} d_q } } \hskip.1cm H^j \chi_{n+j}
\hskip.5cm \text{for all} \hskip.2cm
j \ge 0 .
\end{equation}
Note that $\Vert \xi_{n,j} \Vert = \Vert \chi_{n,j} \Vert$.
We have shown:

\begin{prop}
\label{PropMnIsEll2}
Let the notation be as above. For any $n \ge 0$,
the operator
$$
W_n \, \colon \Vi_n \to \ell^2( \N, \Ui_{n,n} )
\hskip.5cm \text{defined by} \hskip.5cm
W_n \big( ( \xi_{n+j} )_{j \ge 0} \big) = ( \chi_{n+j} )_{j \ge 0}
$$
is a surjective isometry, and
$W_n^* \big( ( \chi_{n+j} )_{j \ge 0} \big) = ( \xi_{n+j} )_{j \ge 0}$.
\end{prop}

Let $n \ge 0$.
We define the \textbf{weighted shift} $S_{\Ui, n}$ on $\ell^2(\N, \Ui_{n,n})$ by
$$
S_{\Ui, n} ( \chi_n, \chi_{n+1}, \chi_{n+2}, \chi_{n+3}, \hdots )
= (0, \sqrt{d_n} \chi_n, \sqrt{d_{n+1}} \chi_{n+1}, \sqrt{d_{n+2}} \chi_{n+2}, \hdots ) .
$$
The operator $S_{\Ui, n}$ is the direct sum of $\dim (\Ui_{n, n})$ copies
of the standard weighted shift $S_n$
defined on the usual sequence space $\ell^2(\N)$ by
\begin{equation}
\label{eqdefSn}
S_n( \lambda_0, \lambda_1, \lambda_2, \lambda_3, \hdots )
= (0, \sqrt{d_n} \lambda_0, \sqrt{d_{n+1}} \lambda_1, \sqrt{d_{n+2}} \lambda_2, \hdots ) .
\end{equation}

\begin{prop}
\label{PropHeqS}
With the notation as above, we have for all $n \ge 0$
$$
W_n H W_n^* = S_{\Ui, n} 
\hskip.5cm \text{and} \hskip.5cm
W_n H^* W_n^* = S_{\Ui, n}^* .
$$
\end{prop}

\begin{proof}
Let $\big( \chi_{n+j} \big)_{j \ge 0} \in \ell^2(\N, \Ui_{n, n} )$.
The vector $W_n^* \big( ( \chi_{n+j} )_{j \ge 0} \big)$
is the vector $\xi$ of~\eqref{Seq}, so that
$$
\begin{aligned}
H W^* \big( ( \chi_{n+j} )_{j \ge 0} \big)
&= H \big( ( \xi_{n+j} )_{j \ge 0} \big)
= H \Bigg( \bigg(
\frac{ \sqrt{ d_{n+j} } } { \sqrt{ \prod_{q=n}^{n+j} d_q } } \hskip.1cm H^j \chi_{n+j}
\bigg)_{j \ge 0} \Bigg)
\\
&= (0, \eta_1, \eta_2, \hdots, \eta_k, \hdots)
\end{aligned}
$$
with
$$
\eta_k = \sqrt{ d_{n+k-1} } \frac{1}{ \sqrt{\prod_{q=n}^{n+k-1} d_q} } H^{k-1} \chi_{n+k-1}
= \sqrt{ d_{n+k-1} } \xi_{n+k-1}
\hskip.5cm \text{for all} \hskip.5cm
k \ge 1 .
$$
Therefore
$$
\begin{aligned}
W_n H W_n^* \big( ( \chi_{n+j} )_{j \ge 0} \big)
& = W_n (0, \eta_1, \eta_2, \hdots, \eta_k, \hdots)
\\
& = W_n 
(0, \sqrt{ d_n} \xi_n, \sqrt{ d_{n+1} } \xi_{n+1}, \sqrt{ d_{n+2} } \xi_{n+2}, \hdots ) ,
\\
& = S_{\Ui, n}
(\chi_n, \chi_{n+1}, \chi_{n+2}, \chi_{n+3}, \hdots) ,
\end{aligned}
$$
hence $W_n H W_n^* = S_{\Ui, n}$.
Finally $W_n H^* W_n^* = \left( W_n H W_n^* \right)^* = S_{\Ui, n}^*$.
\end{proof}

For $n \ge 0$, we denote by
$$
\delta_{*,n}
\hskip.5cm \text{the sequence} \hskip.5cm
(\sqrt{d_n}, \sqrt{d_{n+1}}, \hdots, \sqrt{d_{n+j}}, \hdots )
$$
and we consider the infinite Jacobi matrix
\begin{equation}
\label{eqJacdelta}
J_{\delta_{*,n}} = \begin{pmatrix}
0 & \sqrt{ d_n } & 0 & 0 & \cdots
\\
\sqrt{ d_n } & 0 & \sqrt{ d_{n+1} } & 0 & \cdots
\\
0 & \sqrt{ d_{n+1} } & 0 & \sqrt{ d_{n+2} } & \cdots
\\
0 & 0 & \sqrt{ d_{n+2} } & 0 & \cdots
\\
\cdots & \cdots & \cdots & \cdots & \ddots
\end{pmatrix} 
\end{equation}
If we identify the operators $S_n$ of~\eqref{eqdefSn}
and $S_n^*$ with their matrices
with respect to the standard basis $\left( \delta_j \right)_{j \in \N}$ of $\ell^2(\N)$,
we have
$$
J_{\delta_{*,n}} = S_n + S_n^* .
$$
Here is a reformulation of part of the previous propositions.

\begin{prop}
\label{PropAeqJ}
Let $T = (V, E)$ be an infinite spherically symmetric tree with root $v_0$
and with sequence of branching degrees $(d_r)_{r \ge 0}$
such that $d_r \ge 2$ for all $r \ge 0$ and $\sup_r d_r < \infty$.
\par
The adjacency operator $A_T$ of $T$ is unitarily equivalent to a direct sum
\par\noindent
$\bigoplus_{n = 0}^\infty m_n J_{\delta_{*,n}}$,
where the multiplicities $m_n$ are given by
$$
\begin{aligned}
m_n &= \dim \Ui_{n, n} = \Big( \prod_{q=0}^{n-2} d_q \Big) (d_{n-1} - 1)
\hskip.5cm \text{for} \hskip.2cm
n \ge 2
\\
m_1 &= \dim \Ui_{1, 1} = d_0 - 1
\\
m_0 &= \dim \Ui_{0, 0} = 1
\end{aligned}
$$
and where the $J_{\delta_{*,n}}$~'s
are the Jacobi matrices of~\eqref{eqJacdelta}.
\end{prop}

As a first particular case, consider an integer $d \ge 2$,
the constant sequence $(d, d, d, \hdots)$,
and the \textbf{regular rooted tree $T_d^{{\rm root}} = (V, E)$ of branching degree~$d$};
the relevant Jacobi matrix
is the multiple $\sqrt{d} J$ of the free Jacobi matrix $J$ of Section~\ref{SectionRayLL}.
By Proposition~\ref{infray} for the marked spectrum of $J$
and by Proposition~\ref{rescale},
we obtain the marked spectrum of $\sqrt d J$:
\begin{enumerate}[label=(\arabic*)]
\item
The norm of $\sqrt d J$ is
$2 \sqrt d$.
\item
The spectrum of $\sqrt d J$ is $\mathopen[ -2 \sqrt d , 2 \sqrt d \mathclose]$.
\item
The vertex spectral measure of $\sqrt d J$ at $\delta_0$ is
$d \mu (x) = \frac{1}{2\pi d} \sqrt{ 4d - x^2} \hskip.1cm dx$
for $x \in \mathopen[ -2 \sqrt d , 2 \sqrt d \mathclose]$
(where $dx$ stands for the Lebesgue measure).
\item
The vector $\delta_0$ is cyclic for $\sqrt d J$
and the operator $\sqrt d J$ is multiplicity-free.
\end{enumerate}
By Proposition~\ref{PropAeqJ},
the adjacency operator of $T_d^{{\rm root}}$
is the direct sum of infinitely many copies of $\sqrt d J$,
and we obtain the following:

\begin{prop}
\label{propRootedTInfmult}
Let $d \ge 2$ and let $T_d^{{\rm root}} = (V, E)$ be
the regular rooted tree of branching degree~$d$.
Let $A_d^{{\rm root}}$ denote the adjacency operator of $T_d^{{\rm root}}$.
\begin{enumerate}[label=(\arabic*)]
\item\label{1DEpropRootedTInfmult}
The norm of $A_d^{{\rm root}}$ is
$2 \sqrt{d}$.
\item\label{2DEpropRootedTInfmult}
The spectrum of $A_d^{{\rm root}}$ is
$\mathopen[ -2 \sqrt{d}, 2 \sqrt{d} \mathclose]$.
\item\label{3DEpropRootedTInfmult}
The vertex spectral measure at $0$ is given by
$d \mu (x) = \frac{1}{2\pi d} \sqrt{ 4d - x^2} \hskip.1cm dx$
for $x$ in $\Sigma ( A_d^{{\rm root}} )$;
it is a scalar-valued spectral measure for $A_d^{{\rm root}}$.
\item\label{4DEpropRootedTInfmult}
$A_d^{{\rm root}}$ has uniform infinite multiplicity.
\end{enumerate}
\end{prop}

Recall from the introduction that
two graphs $G, G'$ of bounded degree are \textbf{cospectral}
if their adjacency operators have equal spectra,
equivalent scalar-valued spectral measures,
and spectral multiplicity functions which are equal almost everywhere.

\begin{cor}
\label{ExampleCospectrauxRT}
For any integer $d \ge 2$,
the lattice graph $L_d$
and the regular rooted tree $T_{d^2}^{{\rm root}}$
are cospectral.
\end{cor}

\begin{proof}
This is an immediate consequence
of Corollary~\ref{CorRefUnitEq} and
of Propositions~\ref{lattices}
and~\ref{propRootedTInfmult}.
\par

Note that the measure $\mu_d$ of Proposition~\ref{lattices}
for $L_d$
and the measure $\mu$ of Proposition~\ref{propRootedTInfmult}
for $T_{d^2}^{{\rm root}}$
are not equal,
but they are both equivalent to the Lebesgue measure on $\mathopen[ -d^2, d^2 \mathclose]$,
and this is enough to apply Corollary~\ref{CorRefUnitEq}.
\end{proof}

\begin{exa}
\label{ExSuitedperio}
Consider an integer $p \ge 2$
and a sequence of integers $d_* = (d_r)_{r \ge 0}$ such that
$d_r \ge 2$ and $d_{p+r} = d_r$ for all $r \ge 0$.
For $s \in \{0, 1, \hdots, p-1\}$,
let $T_s$ be the spherically symmetric rooted tree
with sequence of branching degrees $d_{*, s} = (d_s, d_{s+1}, d_{s+2}, \hdots)$.
When $p$ is the smallest period of the sequence $d_*$,
the trees $T_0, \hdots, T_{p-1}$ are pairwise non-isomorphic.
\par

It follows from Proposition~\ref{PropAeqJ}
that the $p$ trees $T_0, \hdots, T_{p-1}$ are cospectral.
\end{exa}

\section{Regular trees}
\label{SectionRegTrees}

For any positive real number $a$, set
\begin{equation}
\label{eqJacAvecA}
J_a =
\begin{pmatrix}
0 & a & 0 & 0 & \cdots
\\
a & 0 & 1 & 0 & \cdots
\\
0 & 1 & 0 & 1 & \cdots
\\
0 & 0 & 1 & 0 & \cdots
\\
\vdots & \vdots & \vdots & \vdots & \ddots
\end{pmatrix} ,
\end{equation}
Note that $J_1$
is the free Jacobi matrix.
Matrices $J_{***}$ here and below are identified
with the corresponding operators on the Hilbert space $\ell^2(\N)$,
with its canonical orthonormal basis.
\par

Let $d$ be an integer, $d \ge 3$; let $T_d = (V, E)$
be the \textbf{regular tree of degree~$d$}.
Choose one vertex $v_0 \in V$ to be the root of $T_d$.
Then $T_d$ is the spherically symmetric rooted tree
with sequence of branching degrees $(d, d-1, d-1, d-1, \hdots)$
of which all terms are $d-1$ but the initial one which is $d$.
The matrix $J_{\delta_{*, 0}}$ of Proposition~\ref{PropAeqJ} is
\begin{equation}
\label{eqJacrootroot}
\begin{aligned}
J_{\sqrt d, \sqrt{d-1}^\infty} &= \begin{pmatrix}
0 & \sqrt d & 0 & 0 & 0 & \cdots
\\
\sqrt d & 0 & \sqrt{d-1} & 0 & 0 & \cdots
\\
0 & \sqrt{d-1} & 0 & \sqrt{d-1} & 0 & \cdots
\\
0 & 0 & \sqrt{d-1} & 0 & \sqrt{d-1} & \cdots
\\
0 & 0 & 0 & \sqrt{d-1} & 0 & \cdots
\\
\vdots & \vdots & \vdots & \vdots & \vdots & \ddots
\end{pmatrix} ,
\\ &=
J_{\sqrt d, \sqrt{d-1}^\infty} = \sqrt{d-1} J_a
\hskip.5cm \text{for} \hskip.5cm
a = \frac{ \sqrt{d} }{ \sqrt{d-1} } .
\end{aligned}
\end{equation}
Note that $1 \le a \le \sqrt{3/2}$, since $d \ge 3$.
The other matrices $J_{\delta_{* n}}$ of Proposition~\ref{PropAeqJ},
for $n \ge 1$, are all equal to $\sqrt{d-1} J_1$.
For Proposition~\ref{propTInfmult} below,
we will need to know properties
of the scalar-valued spectral measures defined by these matrices.
This is straightforward and very standard for $J_1$,
as already shown in Proposition~\ref{infray},
but we did not find a simple ad hoc argument for $J_{\sqrt d / \sqrt{d-1}}$,
and we rather quote the following

\begin{prop}
\label{PropQuoteJa}
Consider a real number $a$ such that $0 < a \le \sqrt{2}$
and the matrix $J_a$ of~\eqref{eqJacAvecA},
viewed as a self-adjoint operator acting on $\ell^2(\N)$,
with its canonical orthonormal basis $(\delta_n)_{n \ge 0}$.
\begin{enumerate}[label=(\arabic*)]
\item\label{1DEPropQuoteJa}
The norm of $J_a$ is $2$.
\item\label{2DEPropQuoteJa}
The spectrum of $J_a$ is the interval $\mathopen[ -2 , 2 \mathclose]$.
\item\label{3DEPropQuoteJa}
The vector $\delta_0$ is cyclic for the operator $J_a$.
\item\label{4DEPropQuoteJa}
The vertex spectral measure of $J_a$
is equivalent to the Lebesgue measure on $\mathopen[ -2 , 2 ]$,
and it is a scalar-valued spectral measure.
\end{enumerate}
\end{prop}

\begin{proof}[Proof for~\ref{1DEPropQuoteJa} to~\ref{3DEPropQuoteJa}
and reference for~\ref{4DEPropQuoteJa}]
As in the proof of Proposition~\ref{PropSpecD},
we have
$\Sigma_{{\rm ess}} (X+K) = \Sigma_{{\rm ess}} (X)$,
so that $\Sigma_{{\rm ess}} (J_a) = \mathopen[ -2, 2 \mathclose]$;
this holds for all $a \ge 0$.
The eigenvalue equation $J_a \xi = \lambda \xi$
for $\xi = (\xi_n)_{n \ge 0} \in \ell^2(\N)$
gives rise to a difference equation of second order with constant coefficients,
and a routine computation shows that this equation has no solution in $\ell^2(\N)$
when $0 < a^2 \le 2$ (details for example in \cite[Lemma 4.6]{BrHN});
it follows that $\Sigma (J_a) = \Sigma_{{\rm ess}} (J_a) = \mathopen[ -2, 2 \mathclose]$.
This completes the proof of Claims~\ref{1DEPropQuoteJa} and~\ref{2DEPropQuoteJa}.
It is straightforward to check Claim~\ref{3DEPropQuoteJa}.
\par

Claim~\ref{4DEPropQuoteJa} is more delicate to prove,
and we quote here a particular case of the result of \cite{MaNe--83}
(particular because we impose diagonal coefficient $b_n = 0$ here,
and because we exclude eigenvalues):

\vskip.2cm

\emph{
Let $(a_n)_{n \ge 0}$ be a sequence of positive real numbers
such that $\lim_{n \to \infty} a_n = 1$ and $\sum_{n=1}^\infty \vert a_{n+1} - a_n \vert < \infty$.
Let $\mu$ be the measure associated to the sequence of orthonormal polynomials $(P_n)_{n \ge 0}$
defined by the recurrence formula
$$
xP_n(x) = a_n P_{n+1}(x) + a_{n-1} P_{n-1}(x)
\hskip.5cm \text{for} \hskip.5cm
n \ge 0
$$
(with ${a_{-1} = 0}$, $P_{-1} = 0$, $P_0$ constant)
and the normalisation $P_n (x) = \gamma_n x^n + \text{lower order terms}$, $\gamma_n > 0$.
Consider the operator $J$ defined on the Hilbert space $\ell^2(\N)$
with its canonical basis $(\delta_n)_{n \in \N}$ by the Jacobi matrix
\begin{equation}
\label{eqJacobiDeDoNe}
\begin{pmatrix}
0 & a_0 & 0 & 0 & \cdots
\\
a_0 & 0 & a_1 & 0 & \cdots
\\
0 & a_1 & 0 & a_2 & \cdots
\\
0 & 0 & a_2 & 0 & \cdots
\\
\vdots & \vdots & \vdots & \vdots & \ddots
\end{pmatrix} ,
\end{equation}
and assume that this operator does not have any eigenvalue.
Let $\mu$ be the local spectral measure of $J$ at $\delta_0$,
defined by
$\int_{\Sigma (J)} f(x) d\mu(x) = \langle f(J) \delta_0 \sca \delta_0 \rangle$
for any function $f$ continuous on the spectrum $\Sigma (J)$ of $J$.
}
\par
\emph{
Then $\Sigma (J) = \mathopen[ -2, 2 \mathclose]$
and $\mu = \rho \lambda$ for a function $\rho$
which is continuous positive on $\mathopen] -2, 2 \mathclose[$
and zero outside $\mathopen[ -2, 2 \mathclose]$
(where $\lambda$ is the Lebesgue measure).
In particular, $\mu$ is equivalent to~$\lambda$ on $\mathopen[ -2, 2 \mathclose]$.
}

\vskip.2cm

Claim~\ref{4DEPropQuoteJa} follows.
Rather than relying on \cite{MaNe--83},
we could alternatively quote \cite[Theorem III.11]{Yafa--17},
which provides an explicit formula for the local spectral measure of $J_a$
at the vector $\delta_0$,
or quote results related to that of \cite{MaNe--83},
such as \cite[Theorem 3]{DoNe--86} or \cite[Theorem 8.18]{Yafa--22}.
\end{proof}

By Corollary~\ref{CorRefUnitEq}, we have the following consequence
of Proposition~\ref{PropQuoteJa}, surprising for us:

\begin{cor}
For any $a \in \mathopen]0, \sqrt 2 \mathclose]$,
the matrix $J_a$ is unitarily equivalent to $J_1$.
\end{cor}

In contrast, for $a > \sqrt 2$, the operator $J_a$ has two simple eigenvalues
$\pm \frac{ a^2 }{ \sqrt{ a^2 - 1 } }$,
and therefore is not unitarily equivalent to $J_1$.
Let $d \ge 3$ and $a = \sqrt{d} / \sqrt{d-1}$;
note that $a < \sqrt 2$;
since $J_{\sqrt d, \sqrt{d-1}^\infty} = \sqrt{d-1} J_a$,
see~\eqref{eqJacrootroot},
Proposition~\ref{PropQuoteJa} implies:
(1)
The norm of $J_{\sqrt d, \sqrt{d-1}^\infty}$ is $2 \sqrt {d-1}$.
(2)
The spectrum of $J_{\sqrt d, \sqrt{d-1}^\infty}$
is the interval $\mathopen[ -2 \sqrt {d-1} , 2 \sqrt {d-1} \mathclose]$.
(3)
The vector $\delta_0$ is cyclic for the operator $J_{\sqrt d, \sqrt{d-1}^\infty}$.
(4)
The vertex spectral measure of $J_{\sqrt d, \sqrt{d-1}^\infty}$
is equivalent to the Lebesgue measure
on $\mathopen[ -2 \sqrt {d-1} , 2 \sqrt {d-1} ]$;
it is a scalar-valued spectral measure.
\par

By Proposition~\ref{PropAeqJ},
the adjacency operator $A_d$ of $T_d$
is the direct sum of
one copy of $J_{\sqrt d, \sqrt{d-1}^\infty}$
and infinitely many copies of $\sqrt{d-1} J_1$,
hence we obtain the following:

\begin{prop}
\label{propTInfmult}
Let $d \ge 3$ and let $T_d = (V, E)$ be
the regular tree of degree~$d$.
Let $A_{T_d}$ be the adjacency operator $T_d$.
\begin{enumerate}[label=(\arabic*)]
\item\label{1DEpropTInfmult}
The norm of $A_{T_d}$ is $2 \sqrt{d-1}$.
\item\label{2DEpropTInfmult}
The spectrum of $A_{T_d}$ is
$\mathopen[ -2 \sqrt{d-1}, 2 \sqrt{d-1} \mathclose]$.
\item\label{3DEpropTInfmult}
The vertex spectral measure at any vertex
is equivalent to the Lebesgue measure on
the spectrum of $A_{T_d}$;
 it is a scalar-valued spectral measure.
\item\label{4DEpropTInfmult}
$A_{T_d}$ has uniform infinite multiplicity.
\end{enumerate}
\end{prop}

\begin{cor}
\label{ExampleCospectrauxT}
For any integer $d \ge 2$,
the lattice graph $L_d$
and the regular tree $T_{d^2 + 1}$
are cospectral.
\end{cor}

Remark: the vertex spectral measures of $T_d$ and $T_d^{{\rm root}}$
which appear here are equivalent to the Lebesgue measure
on the appropriate interval.
This is in sharp contrast with large families of spherically symmetric rooted trees,
for which vertex spectral measures don't have absolutely continuous spectrum
\cite{BrFr--09}, \cite{DaSu--19}.

\section{An uncountable family of cospectral
graphs from \cite{GrNP--22}}
\label{SectionSurGrNP}

There are in \cite{GrNP--22} examples of uncountable families
of pairwise non-isomorphic cospectral Schreier graphs.
They are defined in terms of certain groups of automorphisms
of infinite regular rooted trees called spinal groups,
and the actions of these groups on the boundaries of the trees.
We restrict here to the particular case of the
Fabrykowski--Gupta group,
which is the simplest of the spinal groups
acting on rooted trees of branching degree $\ge 3$,
and we describe shortly one of these families as follows.
\par

Consider the regular rooted tree $T = T_3^{{\rm root}}$ of branching degree $3$,
its boundary $\partial T$
which is the Cantor space $\{ 0, 1, 2 \}^\N$ of infinite sequences of $0, 1$ and $2$~'s,
and the Bernoulli measure $\nu$ on $\partial T$
which is a probability measure invariant by the automorphism group of $T$.
The Fabrykowski--Gupta group $\Gamma$
is the group of automorphisms of $T$ generated by
the symmetric set $S = \{ a, a^{-1}, b, b^{-1} \}$,
where $a$ is the cyclic permutation
of the three main branches of $T$ just below the root,
and where $b$ is the automorphism of $T$
usually defined recursively by $b = (a, 1, b)$,
see for example \cite[Subsection 8.2]{NaPe--21}.
\par

For $\xi \in \partial T$,
let ${\rm Stab}_\xi (\Gamma)$ denote
the stabilizer $\{ g \in \Gamma : g \xi = \xi \}$.
Let ${\rm Sc}_\xi = {\rm Sc} (\Gamma, {\rm Stab}_\xi (\Gamma), S)$
be the \textbf{Schreier graph} of the indicated triple,
with vertex set the orbit $\Gamma \xi$
(i.e., the coset space $\Gamma / {\rm Stab}_\xi (\Gamma)$)
and edges the pairs of the form $\{ g\xi, sg\xi \}$
with $g \in \Gamma$ and $s \in S$.
This graph may have loops (pairs with $g \xi = sg\xi)$)
and multiple edges (pairs $\{ g\xi, sg\xi \}$ and $\{ g\xi, s'g\xi \}$
with $s' \ne s$ and $sg\xi = s'g\xi$),
but its adjacency operator $A_\xi$ acting on $\ell^2(G \xi)$
can be naturally defined.
\par

It is known that there exists
a measurable subset $\mathcal W$ of $\partial T$ of full measure,
i.e., $\nu(\mathcal W) = 1$, such that for $\xi \in \mathcal W$
the adjacency operator $A_\xi$
has the following properties:
\begin{enumerate}
\item[$-$]
The closure of the set of eigenvalues of $A_\xi$,
which is the spectrum of $A_\xi$,
is the union of a Cantor subset of $\R$ of Lebesgue measure zero
and of countably many points accumulating on this Cantor set;
see \cite[Theorem 3.6 \& Corollary 4.13]{BaGr--00}
and \cite[Theorem 1.5]{GrNP--22}.
%
\item[$-$]
$A_\xi$ has a pure point spectrum,
more precisely there exists an orthonormal basis of $\ell^2(\Gamma \xi)$
of eigenvectors of $A_\xi$,
moreover each eigenvector in this basis
is a function of finite support on $\Gamma \xi$
\cite[Theorem 1.8]{GrNP--22}.
\item[$-$]
The set of these eigenvalues 
and their multiplicities, which are all infinite,
do not depend on $\xi$
\cite[Section 5]{GrNP--22}.
\end{enumerate}
Moreover, for $\xi \in \mathcal W$,
the set of $\xi' \in \mathcal W$ for which ${\rm Sc}_{\xi'}$
is isomorphic to ${\rm Sc}_\xi$ has $\nu$-measure $0$
\cite[Corollary 7.13]{NaPe--21}.
\par

In particular, there are uncountably many graphs ${\rm Sc}_\xi$
which are cospectral and pairwise non-isomorphic.

\end{document}